\documentclass[11pt,a4paper]{article}

\usepackage{amsmath, amsfonts, amssymb}
\usepackage{theorem}
\usepackage[dvips]{epsfig}
\usepackage{latexsym}
\usepackage{exscale}
\usepackage[latin1]{inputenc}
\newtheorem{thm}{Theorem}[section]

\newtheorem{prop}[thm]{Proposition}

\newtheorem{defn}{Definition}[section]
\newtheorem{cor}[thm]{Corollary}

\def\be#1\ee{\begin{equation}#1\end{equation}}
\newcommand{\bea}{\begin{eqnarray}}
\newcommand{\eea}{\end{eqnarray}}
\newcommand{\beaa}{\begin{eqnarray*}}
\newcommand{\eeaa}{\end{eqnarray*}}
\newcommand{\bei}{\begin{itemize}}
\newcommand{\eei}{\end{itemize}}
\newcommand{\bee}{\begin{enumerate}}
\newcommand{\eee}{\end{enumerate}}

\def\vp{\varphi}

\def\abs#1{\left\vert #1 \right\vert}      

\def\R{\mathbb{R}}

\def\N{{\mathbb N}}

\def\ex{\mathrm{e}}
\def\d{\, \mathrm{d}}

\newcommand{\eps}{\varepsilon}
\newenvironment{proof}[1][] {\noindent {\bf Proof#1:} }{\hspace*{\fill}$\square$\medskip\par}
\newcommand{\ind}{1\hspace{-0.098cm}\mathrm{l}}

\def\al{{\alpha}}
\def\Ra{R^{\al(\cdot)}}
\begin{document}

\title{\bf Fractional integration operators of variable order:
Continuity and compactness properties}
\author{
Mikhail\ Lifshits \and  Werner\ Linde
}

\date{\today }

\maketitle

\begin{abstract}
Let $\al:[0,1]\to \R$ be a Lebesgue--almost everywhere positive function.
We consider the Riemann--Liouville operator of variable order defined by
\[
  (\Ra f)(t) :=\frac{1}{\Gamma(\al(t))}\,\int_0^t (t-s)^{\al(t)-1}\,f(s)\,\d s
  \,,\quad t\in [0,1]\,,
\]
as operator from $L_p[0,1]$ to $L_q[0,1]$. Our first aim is to study its continuity
properties. For example, we show that $\Ra$ is always bounded (continuous) in
$L_p[0,1]$ provided that $1<p\le\infty$. Surprisingly, this becomes false for $p=1$.
In order  $\Ra$ to be bounded in $L_1[0,1]$, the function $\al(\cdot)$ has to satisfy
some additional assumptions.

In the second, central part of this paper we investigate compactness properties of $\Ra$.
We characterize functions $\al(\cdot)$ for which $\Ra$  is a compact operator and
for certain classes of functions $\al(\cdot)$ we provide
order--optimal bounds for the dyadic entropy numbers $e_n(\Ra)$ .
\end{abstract}
\bigskip
\bigskip
\bigskip
\bigskip
\bigskip

\vfill
\noindent
\textbf{2010 AMS Mathematics Subject Classification:}
Primary: 26A33 Secondary:  47B06, 47B07

\noindent\textbf{Key words and phrases:}\  Riemann--Liouville operator,
integration of variable order, compactness properties, entropy numbers.

\newpage

\section{Introduction}
Different kinds of integration of variable order were introduced in
\cite{SR} "stimulated by intellectual curiosity" with the "hopeful
expectation that applications would follow". Actually, it happened
that just few years later the wide field of applications emerged
independently in probability theory under the name of multifractional
random processes, see \cite{ACL,BJR,FL}, to mention just a few.
These processes are in a natural way related to the integration
operators of variable order.

Subsequent development mainly led to considering these integral
operators in the spaces of varable index, such as Lebesgue spaces
$L_{p(\cdot)}$ and H\"older spaces  $H^{\al(\cdot)}$, see e.g.
\cite{SSV}, \cite{SV}, \cite{S94}, as well as to a theory of
differential equations of variable order.

Since our motivation comes from probability theory, we are not
interested in such elaborated concepts as $L_{p(\cdot)}$ or
$H^{\al(\cdot)}$. Instead, we consider the integration operators
in conventional $L_p$--spaces and study their approximation properties
-- those closely related with the important features of associated
random processes. We are aware about only one work \cite{S08} relating
probability with fractional integration operators of variable order.
However, the operators in \cite{S08} are different from ours and the
emphasis there is put on large time scale properties such as long range
dependence.

To be more precise, let $\al : [0,1]\to\R$ be a measurable function
with $\al(t)>0$ a.e. For a given function $f$ we  define $\Ra f$ by
\be
   \label{Raft}
  (\Ra f)(t)
  :=\frac{1}{\Gamma(\al(t))}\,\int_0^t (t-s)^{\al(t)-1}\,f(s)\,\d s
  \,,\quad t\in [0,1]\,.
\ee
The basic aim of the present paper is to describe properties of $\Ra $,
e.g.~as operator in $L_p[0,1]$, in dependence of those of $\al(\cdot)$.
One of the questions we investigate is in which cases $\Ra$ defines a
bounded (linear) operator from $L_p[0,1]$ into $L_q[0,1]$ for given
$1\le p,q\le\infty$. Recall that in the classical case, i.e., if
$\al(t)=\al$ for some $\al>0$, then this is so provided that
$\al>(1/p -1/q)_+$. Moreover, even in the critical case $\al=1/p -1/q>0$
the classical Riemann--Liouville operator is bounded whenever $1<p<1/\al$.
If $\al(\cdot)$ is non--constant, we shall prove the following:

\begin{thm} \label{tb}
Suppose $1<p\le \infty$. Then for each measurable a.e.~positive
$\al(\cdot)$ the mapping $\Ra$ defined by $(\ref{Raft})$ is a bounded
operator from $L_p[0,1]$ to $L_p[0,1]$. Moreover, the operator norm
of the $\Ra$ is uniformly bounded, i.e., we have
\be \label{Ran}
    \|\Ra : L_p[0,1]\to L_p[0,1]\|\le c_p
\ee
with a constant $c_p>0$ independent of $\al(\cdot)$.
\end{thm}

In contrast to the classical case of constant $\al>0$ it turns out that
Theorem \ref{tb} is no longer valid for $p=1$. We shall give necessary
and sufficient conditions in order that $\Ra$ is bounded in $L_1[0,1]$
as well. Furthermore, we also investigate the question for which
$\al(\cdot)$ equation \eqref{Raft} defines a bounded operator from $L_p[0,1]$
into $L_q[0,1]$.

If $\al(t)=\al$ with $\al> (1/p-1/q)_+$, then $R^\al$ is not only bounded from
$L_p[0,1]$ to $L_q[0,1]$, it even defines a compact operator. Thus another
natural question is whether this is also valid provided that
$\al(t)>(1/p -1/q)_+$ a.e. We investigate this problem  in Section \ref{CpR}.
It turns out that $\Ra$ acts as a compact operator from $L_p[0,1]$ into
$L_q[0,1]$ provided that $\al(\cdot)$ is well separated from the border
value $(1/p-1/q)_+$. What happens if $\al(\cdot)$ approaches the border
value? We investigate this question more thoroughly for $p=q$. The answer is
that $\Ra$ is only compact if $\al(\cdot)$ approaches the critical value
zero extremely slowly.

Suppose now that $\al(\cdot)$ is well--separated from the border value
$(1/p-1/q)_+$. Hence, $\Ra$ is compact and a natural question is how the
degree of compactness of $\Ra$ depends on certain properties of the
underlying function $\al(\cdot)$. We shall measure this degree by the
behavior of the entropy numbers $e_n(\Ra)$. The answer to this question
is not surprising: The degree of compactness of $\Ra$ is "almost" completely
determined by the minimal value of $\al(\cdot)$, i.e.~by the value
$\al_0:=\inf_{0\le t\le 1} \al(t)$. Extra logarithmic terms improve the
behavior of $e_n(\Ra)$ in dependence of the behavior of $\al(\cdot)$ near
to its minimum. For example, if $\al(t)=\al_0 + \lambda t^\gamma$ for some
$\lambda,\gamma>0$ and $\al_0>(1/p-1/q)_+$, then it follows that
$$
  e_n(\Ra : L_p[0,1]\to L_q[0,1])
  \approx \frac{n^{-\al_0}}{(\ln n)^{(\al_0+1/q-1/p)/\gamma}}\,.
$$
Thus, in view of the extra logarithmic term, the entropy behavior of $\Ra$ is slightly
better than $n^{-\al_0}$, the behavior  of $e_n(R^{\al_0})$.
\medskip

To prove entropy estimates for $\Ra$ we have
to know more about the entropy behavior of classical Riemann--Liouville
operators. More precisely, suppose $R^\al$ is the classical operator for some
$\al>(1/p-1/q)_+$. Then it is known that
$$
    C_\al := \sup_{n\ge 1} n^\al\,e_n(R^\al) <\infty
$$
where $R^\al : L_p[0,1]\to L_q[0,1]$. Yet we did not find any information
in the literature how these constants $C_\al$ depend on $\al$. We investigate
this question in Section \ref{Eb}. In particular, we prove that the
$C_\al$ are uniformly bounded for $\al$ in compact sets. The presented
results may be  of interest in their own right because they also sharpen
some known facts about compactness properties of certain Sobolev embeddings.

The organization of the paper is as follows:  In Section \ref{Bp} we first
show that the integral \eqref{Raft} is well--defined for all $f\in L_1[0,1]$
and we state some weak form of a semi--group property for $\Ra$.
Section \ref{Bound} is devoted to the question of boundedness of $\Ra$.
More precisely, we prove the above stated Theorem \ref{tb} and also
characterize functions $\al(\cdot)$ for which $\Ra$ defines a
bounded operator from $L_p[0,1]$ into $L_q[0,1]$. Let $0<r<\infty$
be a given real number and let $\al(\cdot)$ be a function on $[0,r]$
possessing a.e.~positive values. Then $\Ra f$ is well--defined for
$f\in L_p[0,r]$. We investigate in Section \ref{Sp} how $\Ra$ on
$L_p[0,r]$ may be transformed into an operator defined on $L_p[0,1]$.
In Section \ref{CpR} we characterize functions $\al(\cdot)$ for which
$\Ra$ is a {\it compact} operator in $L_p[0,1]$. Here we distinguish
between the two following cases: Firstly, the function $\al(\cdot)$
approaches the border value at zero and, secondly, the critical value
of $\al(\cdot)$ appears at the right hand end point of $[0,1]$. As
already mentioned, in order to investigate compactness properties of
$\Ra$ we have to know more about those of classical Riemann--Liouville
operators. We present the corresponding evaluations in Section \ref{Eec}.
Starting with some general upper and lower entropy estimates for $\Ra$, which are
presented in Section \ref{Eb},
we obtain in Section \ref{Exa} sharp estimates for the
entropy numbers $e_n(\Ra)$ for concrete functions $\al(\cdot)$.
\bigskip

\noindent
\textbf{Acknowledgement:} The authors are very grateful to Thomas K\"uhn
(Leipzig University) for very helpful discussions about Proposition \ref{pr1}.
He sketched an independent proof of \eqref{anI} and some of his
ideas we incorporated in the proof presented here. Furthermore we thank
Hermann K\"onig (Kiel University) who indicated to us another direct approach
(without using Besov spaces) for estimating $a_n(R^\al)$ uniformly.\\
The research was supported by the RFBR--DFG grant 09-01-91331 "Geometry and
asymptotics of random structures". Furthermore, the first named author was
supported by RFBR grants 10-01-00154a and 11-01-12104-ofi-m.

\section{Basic properties of $\Ra$}
\label{Bp}
\setcounter{equation}{0}

Throughout this paper we always assume that $\al : [0,1]\to [0,\infty)$ is a measurable
function satisfying $\al(t)>0$ for (Lebesgue) almost all $t\in [0,1]$.
Our first aim is to show that the generalized fractional integral \eqref{Raft} exists a.e.
Before let us introduce the following notation used throughout this paper: We set
\be
\label{defK}
K_0:= \inf_{0<t<\infty} \Gamma(t)\approx 0.8856031944\ldots\,.
\ee

\begin{prop}
\label{exist}
For $f\in L_1[0,1]$ the function
$$
(\Ra f)(t) :=\frac{1}{\Gamma(\al(t))}\,\int_0^t (t-s)^{\al(t)-1}\,f(s)\,\d s
$$
is well-defined for almost all $t\in[0,1]$.
\end{prop}

\begin{proof}
For $\beta>0$ define the level sets $A_\beta$ of $\al(\cdot)$ by
$$
A_\beta:=\{t\in [0,1] : \al(t)\ge \beta\}\,.
$$
Then, if $f\in L_1[0,1]$, $f\ge 0$, it follows that
\beaa
&&\int_{A_\beta}\left[\frac{1}{\Gamma(\al(t))}\int_0^t (t-s)^{\al(t)-1}\, f(s) \d s\right] \d t
\le
\int_{A_\beta}\left[\frac{1}{K_0}\int_0^t (t-s)^{\beta-1} f(s) \d s\right] \d t \\
&\le&
\frac{1}{K_0} \int_0^1\left[\int_s^1 (t-s)^{\beta-1}  \d t\right] f(s) \d s
\le\frac{\|f\|_1}{\beta K_0}
\eeaa
with $K_0>0$ defined by \eqref{defK}.
Hence, whenever $f\in L_1[0,1]$, then $(\Ra f)(t)$ exists for almost all $t\in A_\beta$.
Consequently, taking a sequence $(\beta_n)_{n\ge 1}$ tending monotonously to zero, $(\Ra f)(t)$ is
well-defined for almost all $t\in \bigcup_{n=1}^\infty A_{\beta_n}$. Moreover, by $\al(t)>0$ a.e.~the set
$\bigcup_{n=1}^\infty A_{\beta_n}$ possesses Lebesgue measure $1$, and this completes the proof.
\end{proof}

\begin{defn}
\rm
Given $f\in L_1[0,1]$, the function $\Ra f$ is called the Riemann--Liouville fractional integral of
$f$ with varying exponent $\al(\cdot)$. In the case of real $\al>0$ we denote by $R^\al f$
the
classical $\al$--fractional integral of $f$ (in the sense of Riemann--Liouville) which
corresponds to $\Ra f$ with $\al(t)=\al,\; 0\le t\le 1$.
\end{defn}

One of the most useful properties of the scale of classical Riemann--Liouville integrals
is that it possesses a semi-group property in the following sense: Whenever $\al,\beta>0$,
then we have
$$
    R^{\al+\beta} = R^\al\circ R^\beta\,.
$$
In the case of non-constant $\al(\cdot)$ and $\beta(\cdot)$ such a nice rule is no longer
valid. Only the following weaker result holds:
\begin{prop}
\label{sg}
{\rm (\cite{SR}, Theorem 2.4)}
For any $\al(\cdot)$ and any $\beta>0$ we have
$$
    R^{\al(\cdot)+\beta}=\Ra \circ R^\beta\,.
$$
\end{prop}
\begin{proof}
The proof is exactly as in the case of real $\al>0$. Therefore we omit it.
\end{proof}

\begin{remark}
As already mentioned in \cite{SR}, neither $R^\beta\circ \Ra=R^{\al(\cdot)+\beta}$ nor
$\Ra\circ R^{\beta(\cdot)}= R^{\al(\cdot)+\beta(\cdot)}$ are valid in general.
\end{remark}

\section{Boundedness properties of $\Ra$}
\setcounter{equation}{0}
\label{Bound}

\subsection{$\Ra$ as an operator in $L_p[0,1]$ , $1<p\le\infty$}
In the case of real $\al>0$   by (\ref{Raft}) a bounded linear operator $R^\al$
from $L_p[0,1]$ into $L_p[0,1]$ is defined. Moreover, as easily can be seen (cf.~also \cite{AG})
it holds
$$
\|R^\al : L_p[0,1]\to L_p[0,1]\|\le \frac{1}{\Gamma(\al+1)}\le \frac{1}{K_0}
$$
for all $\al>0$ and $1\le p\le \infty$.
Thus it is natural to ask whether or not $\Ra$ defines also a bounded operator in $L_p[0,1]$
for non-constant functions
$\al(\cdot)$. The answer to this question depends on the number $p$.  The positive result
was stated in Theorem \ref{tb}. Our next aim is to prove it.

\noindent
\textbf{Proof of Theorem \ref{tb}:}
The case $p=\infty$ easily follows by
$$
  |(\Ra f)(t)|\le \frac{1}{\Gamma(\al(t))}\int_0^t (t-s)^{\al(t)-1}\d s\cdot \|f\|_\infty
  \le
  \frac{\|f\|_\infty}{\Gamma(\al(t)+1)}\le \frac{\|f\|_\infty}{K_0}\,.
$$

Suppose now $1<p<\infty$.
For each $f\in L_p[0,1]$ its maximal function $M f$ (cf.~\cite{SW}, II (3)) is defined by
$$
     (M f)(t):= \sup_{r>0}\frac{1}{2 r}\int_{-r}^r |f(t-v)| \d v\,,\quad t\in\R\,.
$$
Hereby we extend $f$ to $\R$ by $f(t)=0$ whenever $t\notin [0,1]$. The basic property of $M f$
is that it fulfills the so-called Hardy--Littlewood maximal inequality (cf.~\cite{SW},
Theorem 3.7., Chapter II), asserting that for each $p>1$ there is an $A_p>0$ such that
\be \label{HL}
     \|M f\|_p\le A_p\,\|f\|_p\,,\quad f\in L_p[0,1]\,.
\ee
To proceed, choose $f\in L_p[0,1]$ with $f\ge 0$ and a number $t\in[0,1]$ for which simultaneously
$\al(t)>0$ as well as $(M f)(t)<\infty$ hold. Note that the set of those numbers $t$ possesses
Lebesgue measure $1$. To simplify the notation let us write $\al$ instead of $\al(t)$ for a moment.
Then we get
\be \label{m1}
   (\Ra f)(t) =\frac{1}{\Gamma(\al)} \int_0^t s^{\al -1}\ f(t-s)\, \d s
   =\frac{1}{\Gamma(\al)}\int_0^t s^{\al -1}\d\mu_t(s)
\ee
where $\mu_t$ is the Borel measure on $[0,t]$ with density $v\mapsto f(t-v)$. Let
$$
F_t(s):= \mu_t([0,s]) =\int_0^s f(t-v)\d v\,,\quad 0\le s\le t\,,
$$
then it follows that
\bea \label{m2}
   \nonumber
   \int_0^t s^{\al -1}\d\mu_t(s) &=& \int_0^t (\al-1) s^{\al-2}\mu_t([s,t])\,\d s
   \\  \nonumber
   &=& \int_0^t (\al-1) s^{\al-2}\left[F_t(t) -F_t(s)\right]\,\d s
   \\
   &=& t^{\al -1} F_t(t) - (\al-1) \int_0^t s^{\al -2} F_t(s)\d s\,.
\eea
By the definition of $M f$ it holds
\be
\label{m3}
    F_t(s)\le 2\,(M f)(t)\cdot s\,,\quad 0\le s\le t\,,
\ee
which, in particular, implies that the right hand integral in \eqref{m2}
is finite by the choice of $t$.

Let us first investigate the case $\al\ge 1$. Then (\ref{m1}), (\ref{m2}),
and (\ref{m3}) imply
\be \label{m4}
    (\Ra f)(t)\le \frac{t^{\al -1}}{\Gamma(\al)} 2\, t\, (M f)(t)
    = \frac{2t^{\al}\,(M f)(t)}{\Gamma(\al)}\,.
   \ee

Now, if $0<\al<1$, then (\ref{m1}), (\ref{m2}) and (\ref{m3}) lead to
\bea \label{m5}
\nonumber
    (\Ra f)(t)&\le& \frac{2\,t^{\al}}{\Gamma(\al)} \, (M f)(t)
    + \frac{2\,(1-\al)}{\Gamma(\al)}\cdot \int_0^t s^{\al -1} \d s \cdot (M f) (t)
    \\
    &=& \frac{2 \,t^\al}{\Gamma(\al +1)} (M f)(t).
\eea
Combining (\ref{m4}) and (\ref{m5}) we see that there is a universal $c>0$ (we may choose $c=2/K_0$)
such that for almost all $t\in [0,1]$
$$
  (\Ra f)(t)\le c\, t^{\al(t)} (M f)(t)\le c\,(M f)(t)\,.
$$
Consequently, since $p>1$, we may apply  (\ref{HL}) and obtain
$$
\|\Ra f\|_p \le c\, A_p\,\|f\|_p
$$
for all non-negative functions $f\in L_p[0,1]$.

Finally, if $f\in L_p[0,1]$ is arbitrary, we argue as follows:
$$
\|\Ra f\|_p \le \|\Ra(|f|)\|_p\le c\,A_p\,\| |f|\|_p = c\,A_p\,\| f\|_p
$$
and this completes the proof. Note that the last estimate yields
$$
\|\Ra : L_p[0,1]\to L_p[0,1]\| \le c\, A_p\,,
$$
hence also \eqref{Ran}
with $c_p= c\,A_p$.
{\hspace*{\fill}$\square$\medskip\par}
\medskip

\begin{remark} The idea to use maximal function as an estimate for fractional
integrals appeared earlier in \cite{AS}, for a different purpose.
\end{remark}

\subsection{$\Ra$ as an operator in $L_1[0,1]$}

Inequality (\ref{HL}) fails for $p=1$. Therefore the previous proof does not extend
to that case and it remains unanswered whether Theorem \ref{tb} is valid for $p=1$.
We will prove that the answer is negative, i.e., there are measurable $\al(\cdot)$,
a.e.~positive, such that $\Ra$ is \textit{not} bounded in $L_1[0,1]$.

Before doing so, let us mention that (\ref{HL}) has the following weak type extension
to $p=1$ (cf. \cite{SW}, Theorem 3.4, Chapter II): There is a constant $c>0$ such that
for all $f\in L_1[0,1]$ we have
$$
    |\{t\in [0,1] : (M f)(t)\ge u\}|\le c\,\frac{\|f\|_1}{u}\,,\quad u>0\,,
$$
where, as usual, $|A|$ denotes the Lebesgue measure of a set $A\subseteq \R$. By the
methods developed in the proof of Theorem \ref{tb} this yields the following:

\begin{prop} \label{wL1}
There is a universal $c>0$ such that for all measurable, a.e.~positive
$\al(\cdot)$ we have
$$
    |\{t\in [0,1] : |(\Ra f)(t)|\ge u\}|\le c\,\frac{\|f\|_1}{u}\,,\quad u>0\,.
$$
In different words, $\Ra$ is a bounded operator from $L_1[0,1]$ into the Lorentz
space $L_{1,\infty}[0,1]$.
\end{prop}

The next result gives a first description of functions $\al(\cdot)$ for which $\Ra$ acts as a bounded operator
in $L_1[0,1]$.
\begin{prop} \label{bL1}
Given $\al(\cdot)$ as before, the mapping $\Ra$ is bounded in $L_1[0,1]$ if and only if
\be
\label{L1n}
     \sup_{0\le s\le 1} \int_s^1 \frac{1}{\Gamma(\al(t))} (t-s)^{\al(t)-1} \d t
     <\infty\,.
\ee
Moreover, in this case
$$
    \|\Ra : L_1[0,1]\to L_1[0,1]\|
    = \sup_{0\le s\le 1} \int_s^1 \frac{1}{\Gamma(\al(t))} (t-s)^{\al(t)-1} \d t\,.
$$
\end{prop}

\begin{proof}
Suppose first that (\ref{L1n}) holds. Given $f\in L_1[0,1]$ it follows that
\beaa
  \|\Ra f\|_1 &=&\int_0^1 \abs{\frac{1}{\Gamma(\al(t))} \int_0^t (t-s)^{\al(t)-1}\,f(s)\d s}\d t
  \\
  &\le& \int_0^1 \left[\int_s^1\frac{1}{\Gamma(\al(t))}(t-s)^{\al(t)-1}\d t\right] |f(s)|\d s
  \\
  &\le& \sup_{0\le s\le 1} \int_s^1 \frac{1}{\Gamma(\al(t))}(t-s)^{\al(t)-1} \d t\, \|f\|_1\,.
\eeaa
Hence, $\Ra$ is bounded and
$\|\Ra\|\le \sup_{0\le s\le 1} \int_s^1 \frac{1}{\Gamma(\al(t))}  (t-s)^{\al(t)-1} \d t$ .

Conversely, if $\Ra$ is bounded in $L_1[0,1]$, then
\beaa
   \sup_{0\le s\le 1} \int_s^1  \frac{1}{\Gamma(\al(t))} (t-s)^{\al(t)-1} \d t
   &=& \lefteqn{{\rm ess\, sup}_{0\le s\le 1}
   \int_s^1 \frac{1}{\Gamma(\al(t))}  (t-s)^{\al(t)-1} \d t}
   \\
   &=&   \sup_{\|f\|_1\le 1, f\ge 0}
   \int_0^1 \left[\int_s^1\frac{1}{\Gamma(\al(t))}(t-s)^{\al(t)-1}\d t\right] f(s)\d s
   \\
   &=& \sup_{\|f\|_1\le 1, f\ge 0} \|\Ra f\|_1\le \|\Ra\|\,,
\eeaa
and this completes the proof.
\end{proof}
\medskip

\begin{remark}
In particular, Proposition \ref{bL1} implies that $\Ra$ is bounded as
operator in $L_1[0,1]$ if $\al(\cdot)$ is separated from zero, i.e., if
$\inf_{0\le t\le 1}\al(t)>0$. Of course, this may also be proved directly by the estimates given
in the proof of Proposition
\ref{exist}.
\end{remark}

\begin{cor}
For bounded functions $\al(\cdot)$ condition $(\ref{L1n})$ is equivalent to
\be
\label{L1nn}
\sup_{0\le s\le 1} \int_s^1 \al(t)\,(t-s)^{\al(t)-1} \d t<\infty\,.
\ee
\end{cor}
\begin{proof}
This is a direct consequence of Proposition \ref{bL1} and
$$
K_0\le \Gamma(\al(t)+1)\le \max\{1, \Gamma(\al_1 +1)\}
$$
provided that $\sup_{0\le t\le 1}\al(t)=\al_1<\infty$ .
\end{proof}

\begin{prop} \label{nc}
Let
\be \label{al1}
    \al(t)=\left\{
           \begin{array}{ccc}
                   \frac{1}{|\ln t|}&:& 0<t\le \ex^{-1}\\
                    1 &:& \ex^{-1}\le t\le 1\,,
            \end{array}
            \right.
\ee
Then $\Ra$ is \emph{not} bounded in $L_1[0,1]$.
\end{prop}
\begin{proof}
We show that (\ref{L1nn}) is violated. This follows by
\beaa
    \sup_{0\le s\le 1} \int_s^1 \al(t)\,(t-s)^{\al(t)-1} \d t&\ge&
    \lim_{s\to 0}\int_s^{\ex^{-1}} \al(t)\,(t-s)^{\al(t)-1}\,\d t\\
    &\ge&
    \lim_{s\to 0} \int_s^{\ex^{-1}} \frac{1}{t\,|\ln t|}\,\ex^{-1}\,\d t =\infty\,.
\eeaa
Hence $\Ra$ cannot be bounded in $L_1[0,1]$.
\end{proof}
\medskip

\begin{remark}
In view of Proposition \ref{wL1}, the Closed Graph Theorem implies the following:
If $\al(\cdot)$ is as in (\ref{al1}), then there are functions $f\in L_1[0,1]$ such that
$\Ra f\notin L_1[0,1]$.
\end{remark}
\medskip

For concrete $\al(\cdot)$ condition (\ref{L1n}) might be difficult to verify. Therefore
we are interested in criteria which are easier to handle. Fortunately, under a weak
additional regularity assumption for $\al(\cdot)$ such a criterion exists.

\begin{prop} 
Assume that $\al(\cdot)$ is bounded and satisfies the following regularity condition:
$\exists\, c_1, c_2>0$ such that
\begin{equation} \label{regal}
      c_1 \al(s) \le \al(t)\le c_2\al(s), \qquad s\le t\le \min\{2s,1\},\; 0\le s\le 1\,.
\end{equation}
Then the operator $\Ra$ is bounded in  $L_1[0,1]$ if and only if
\begin{equation} \label{int}
     \int_0^1 \al(t) t^{\al(t)-1} dt <\infty
\end{equation}
holds.
\end{prop}

\begin{proof}
Assume first that \eqref{int} holds.
We will check that the expression in \eqref{L1n} is finite.
For any $s\in[0,1]$ we have
$$
  \int_s^1     \frac{1}{\Gamma(\al(t))} \ (t-s)^{\al(t)-1} \d t
 =
 \left( \int_s^{\min(2s,1)}
 + \int_{\min(2s,1)}^1  \right)  \frac{1}{\Gamma(\al(t))} \ (t-s)^{\al(t)-1} \d t\,.
$$
For the first integral by \eqref{regal} we obtain
\begin{eqnarray*}
   \int_s^{\min(2s,1)} \frac{1}{\Gamma(\al(t))} \  (t-s)^{\al(t)-1} \d t
   &\le& \frac{1}{K_0}\,\int_s^{\min(2s,1)}  \al(t)(t-s)^{\al(t)-1} \d t
\\
   &\le& \frac{1}{K_0}\, \int_s^{\min(2s,1)}  c_2 \al(s)(t-s)^{c_1\al(s)-1} \d t
\\
   &\le& \frac{1}{K_0}\, c_2\, \al(s)  \int_s^{2s} (t-s)^{c_1\al(s)-1} \d t \le \frac{c_2}{K_0\,c_1}\,.
\end{eqnarray*}
To estimate the second integral we use \eqref{int} which implies
\begin{eqnarray*}
   \int_{\min(2s,1)}^1   \frac{1}{\Gamma(\al(t))} \    (t-s)^{\al(t)-1} \d t
   &\le&  \frac{1}{K_0}\, \int_{\min(2s,1)}^1   \al(t)(t-s)^{\al(t)-1} \d t
   \\
   &=&  \frac{1}{K_0}\,\int_{\min(2s,1)}^1   \al(t) t^{\al(t)-1}
        \left(\frac{t-s}{t}\right)^{\al(t)-1} \d t
   \\
   &\le&  \frac{2}{K_0}\, \int_{\min(2s,1)}^1   \al(t) t^{\al(t)-1}  \d t
   \\
   &\le& \frac{2}{K_0}\,\int_{0}^1   \al(t) t^{\al(t)-1}   \d t <\infty.
\end{eqnarray*}
Thus the supremum in \eqref{L1n} is finite and we see that operator $\Ra$
is bounded in $L_1[0,1]$.
\medskip

Conversely, assume that operator $\Ra$ is bounded in $L_1[0,1]$.
Then the supremum in  \eqref{L1n} is finite and by Fatou's lemma and by Proposition \ref{bL1} it follows that
\[
  \int_0^1   \frac{1}{\Gamma(\al(t))}  t^{\al(t)-1} \d t
  \le \liminf_{s\to 0}  \int_s^1 \frac{1}{\Gamma(\al(t))}
  (t-s)^{\al(t)-1} \d t \le ||\Ra|| <\infty\,.
\]
This completes the proof.
\end{proof}

\subsection{$\Ra$ as an
operator from $L_p[0,1]$ to $L_q[0,1]$}

The aim of this subsection is to investigate the following question:
Given $1\le p,q\le\infty$, for which $\al(\cdot)$ is $\Ra$ bounded
from $L_p[0,1]$ into $L_q[0,1]$ ? Let us first recall the answer to
this question in the case of real $\al>0$ (cf.~Theorem 3.5 in \cite{SKM} and Theorem 383 in \cite{HarLit}).
\begin{prop}
\label{tbc}
If $\al>(\tfrac 1 p -\frac 1 q)_+$, then $R^\al$ is bounded
from $L_p[0,1]$ into $L_q[0,1]$. Moreover, if $1<p< 1/\al$, then $R^\al$
is also bounded from $L_p[0,1]$ into $L_q[0,1]$ in the borderline case
$1/q = 1/p -\al$.
\end{prop}

For variable functions $\al(\cdot)$ we have the following result.
\begin{prop} \label{p1_pq}
 Let $p>1$ and $q<\infty$.
 For any $\al(\cdot)$  satisfying $\al(t)>(\tfrac 1p-\tfrac 1q)_+$ for almost all
 $t\in[0,1]$, the operator $\Ra$ is bounded from $L_p[0,1]$ into $L_q[0,1]$ .
\end{prop}

\begin{proof}
If $p\ge q$, then we have $(\tfrac 1p-\tfrac 1q)_+=0$. Thus let us take any $\al(\cdot)>0$ a.e. The operator
$\Ra :L_p[0,1]\to L_q[0,1]$ may be considered as composition of
$\Ra :L_p[0,1]\to L_p[0,1]$ with the (bounded) embedding from $L_p[0,1]$ into $L_q[0,1]$.
Theorem \ref{tb} applies to $\Ra$ as operator in $L_p[0,1]$ (recall that we assume $p>1$) and we obtain the boundedness
of the composition, hence of $\Ra$ from $L_p[0,1]$ into $L_q[0,1]$.

If $p<q$, set $\beta:=1/p - 1/q$ , hence by assumption $\al(t)>\beta$ a.e.
In view of Proposition \ref{sg} we may write $R^\al :L_p[0,1]\to L_q[0,1]$ as the
composition of $R^{\al(\cdot)-\beta} :L_q[0,1]\to L_q[0,1]$ with
$R^\beta :L_p[0,1]\to L_q[0,1]$. Proposition \ref{tbc} yields the boundedness of $R^\beta$
(recall that our assumption $q<\infty$ guarantees that $p<1/\beta$),  while
Theorem \ref{tb} applies to  $R^{\al(\cdot)-\beta}$ and we obtain the boundedness
of the composition.
\end{proof}
\medskip

Finally, let us briefly dwell on the case $q=\infty$ excluded in the previous
proposition. Assuming $\al(\cdot)>\tfrac 1p$ a.e., that is necessary anyway,
in this case we have
\beaa
   ||\Ra :L_p[0,1]\to L_\infty[0,1]||
   &=& \sup_{||f||_p\le 1} \sup_{0\le t\le 1}
   \left| \frac{1}{\Gamma(\al(t))} \int_0^t (t-s)^{\al(t)-1}f(s)\d s \right|
\\
   &=& \sup_{0\le t\le 1} \sup_{||f||_p\le 1}
   \left| \frac{1}{\Gamma(\al(t))} \int_0^t (t-s)^{\al(t)-1}f(s)\d s \right|
\\
   &=& \sup_{0\le t\le 1} \frac{1}{\Gamma(\al(t))}
   \left[  \int_0^t (t-s)^{(\al(t)-1)p'}\d s \right]^{1/p'}
\\
   &=& \frac{1}{(p')^{1/p'}} \sup_{0\le t\le 1} \frac{1}{ \Gamma(\al(t))}
   \frac{ t^{\al(t)-1/p}} {(\al(t)-1/p)^{1/p'}}\,.
\eeaa
Hence, the necessary and sufficient condition for boundedness of $\Ra$ is
\[
    \sup_{0\le t\le 1}
    t^{\al(t)-1/p} (\al(t)-1/p)^{-1/p'} <\infty.
\]
In particular, if $\Ra :L_p[0,1]\to L_\infty[0,1]$ is bounded, then
$\al(\cdot)$ is uniformly separated from $\tfrac 1p$ outside  any neighborhood of zero.

\section{Scaling properties}
\setcounter{equation}{0}
\label{Sp}

The aim of this section is as follows: For a number $r>0$ and a function $\al(\cdot)$
on $[0,r]$ being a.e.~positive we regard $\Ra$ as operator from $L_p[0,r]$ into
$L_q[0,r]$, i.e., given $f\in L_p[0,r]$ it holds
$$
    (\Ra f)(t)
    :=\frac{1}{\Gamma(\al(t))}\int_0^t (t-s)^{\al(t)-1} f(s)\,\d s\,,\quad 0\le t\le r\,.
$$
The question is now, how the operator $\Ra$ acting on $[0,r]$ may be transformed into
a suitable $R^{\tilde\al(\cdot)}$ acting on $[0,1]$ .
\medskip

To answer this, let us introduce the following notation:
For $1\le p\le \infty$ we define the isometry $J_p : L_p[0,1]\to L_p[0,r]$ by
$$
   J_p f(s):= r^{-1/p}\,f(s/r)\,,\quad 0\le s\le r\,,
$$
with the obvious modification for $p=\infty$. Furthermore, we introduce a function
$\tilde\al(\cdot)$ on $[0,1]$ by
$$
    \tilde\al(t) := \al(r\,t)\,,\quad 0\le t\le 1\,,
$$
and, finally, a multiplication operator  $M_{\al,r}$ by
$$
   (M_{\al,r} g)(t):= r^{\tilde\al(t)+1/q-1/p}\cdot g(t)\,,\quad 0\le t\le 1\,.
$$

Now we are in position to state and to prove the announced scaling property of $\Ra$.

\begin{prop} \label{scale}
It holds
$$
   J_q^{-1}\circ \Big[\Ra : L_p[0,r]\to L_q[0,r]\Big]\circ J_p
   = M_{\al,r}\circ \Big[R^{\tilde\al(\cdot)} : L_p[0,1]\to L_q[0,1] \Big]\,.
$$
\end{prop}

\begin{proof}
Given  a function $f\in L_p[0,1]$ elementary calculations lead to
\beaa
   \Ra(J_p f)(t)&=& r^{-1/p} \,\frac{1}{\Gamma(\al(t))}\int_0^t(t-s)^{\al(t)-1}\,f(s/r)\,\d s
   \\
   &=& r^{1 -1/p} \,\frac{1}{\Gamma(\al(t))}\int_0^{t/r}(t-r s)^{\al(t)-1}\,f(s)\,\d s
   \\
   &=& r^{\tilde\al(t/r) -1/p} \,\frac{1}{\Gamma(\tilde\al(t/r))}
       \int_0^{t/r}(t/r -s)^{\tilde\al(t/r)-1}\,f(s)\,\d s
   \\
   &=&(J_q\circ M_{\al,r}\circ R^{\tilde\al(\cdot)} f)(t)\,,\quad 0\le t\le r\,.
\eeaa
This completes the proof.
\end{proof}

\begin{cor}
\label{sc1}
If $1<p\le\infty$, there is a universal $c_p>0$ such that for all $r\in(0,1]$ we have
\be \label{Rar}
    \|\Ra : L_p[0,r]\to L_p[0,r]\|\le c_p\,r^{\al_0}\,.
\ee
Here $\al_0$ is defined by $\al_0:=\inf_{0\le t\le r}\al(t)$ .
\end{cor}

\begin{proof}
By Proposition \ref{scale} and Theorem \ref{tb} it follows that
\beaa
   \|\Ra : L_p[0,r]\to L_p[0,r]\|
   &\le& \|M_{\al,r}: L_p[0,1]\to L_p[0,1]\|\cdot\|R^{\tilde\al(\cdot)}:L_p[0,1]\to L_p[0,1]\|
   \\
   &\le& \sup_{0\le t\le 1} r^{\al(t)}\,c_p = c_p\,r^{\al_0}
\eeaa
because of $0<r\le 1$. This completes the proof.
\end{proof}

\begin{remark}
The preceding result can be easily extended to arbitrary intervals of length
less than one. More precisely, given real numbers $a<b\le a+1$ and a function
$\al(\cdot)$ on $[a,b]$, a.e.~positive, for any $p>1$ it follows
\be \label{sc2}
   \|\Ra : L_p[a,b]\to L_p[a,b]\|\le c_p\,(b-a)^{\al_0}
\ee
with $\al_0=\inf_{a\le t\le b}\al(t)$. Note that $\Ra$ acts on $L_p[a,b]$ as
$$
(\Ra f)(t) =\frac{1}{\Gamma(\al(t))}\int_a^t (t-s)^{\al(t)-1} f(s)\,\d s\,,\quad a\le t\le b\,.
$$
\end{remark}

\section{Compactness properties of $\Ra$}
\label{CpR}
\setcounter{equation}{0}

In Section \ref{Bound} we investigated the the question whether or not $\Ra$ defines a bounded operator from $L_p[0,1]$
into $L_q[0,1]$. But it is not clear at all whether this operator
is even compact, i.e., whether it maps bounded sets in $L_p[0,1]$ into relatively compact subsets
of $L_q[0,1]$. Recall that in the classical case this is so for $1\le p,q\le \infty$ provided that $\al>(1/p-1/q)_+$ (cf.~Proposition \ref{enr} below).
\medskip

We start with an easy observation about the compactness of $\Ra$ for non--constant $\al(\cdot)$.

\begin{prop} \label{comp}
Suppose $\al_0:=\inf_{t\in[0,1]}\al(t)>(1/p-1/q)_+$. Then for all $1\le p,q\le \infty$ the operator $\Ra$
is compact from $L_p[0,1]$ into $L_q[0,1]$.
\end{prop}

\begin{proof}
Choose a number $\beta$ with $(1/p-1/q)_+<\beta<\al_0$. By Proposition \ref{sg} we may represent $\Ra$ as
\be \label{pro}
   \Ra = R^{\al(\cdot)-\beta}\circ R^\beta\
\ee
where $R^\beta$ maps $L_p[0,1]$ into $L_q[0,1]$ and  $R^{\al(\cdot)-\beta}$ acts in $L_q[0,1]$.
Now $R^\beta$ is compact and $R^{\al(\cdot)-\beta}$ is bounded. The latter is a consequence of
$\inf_{0\le t\le 1}[\al(t)-\beta]>0$. Indeed, if $q>1$, the boundedness of $R^{\al(\cdot)-\beta}$
follows by Theorem \ref{tb}. For $q=1$ we may apply  Proposition \ref{bL1} to $\al(\cdot)-\beta$.
Using \eqref{pro} and the ideal property of the class of compact operators, $\Ra$ is compact as
well.
\end{proof}

In particular, the preceding proposition tells us that $\Ra$ is a compact operator in $L_p[0,1]$ provided that $\al(\cdot)$ is well separated from zero. On the other hand, as is well-known (cf.~\cite{SKM}, Theorems 2.6 and 2.7), whenever $f\in L_p[0,1]$,
then it follows that
$$
\lim_{\al\to 0}\|R^\al f -f\|_p=0
$$
as well as
$$
\lim_{\al\to 0} (R^\al f)(t)= f(t)
$$
for almost all $t\in[0,1]$.
In different words, for small $\al>0$ the operator $R^\al$ is close to the non-compact identity operator in $L_p[0,1]$.
Thus it is not clear at all whether $\Ra$ is compact when we drop the assumption
$\inf_{t\in[0,1]}\al(t)>0$.
The aim of this section is to show that $\Ra$ is compact if $\al(\cdot)$ approaches zero
very slowly while it is non-compact if $\al(t)$ is already quite close to zero in a neighborhood
of a critical point, i.e., near to a point where $\al(\cdot)$ approaches zero.
\medskip

We will investigate the two following cases separately:
\bee
\item
It holds $\inf_{\eps\le t\le 1}\al(t)>0$ for each $\eps>0$, i.e., the critical point of
$\al$ is $t=0$.
\item
The critical point of $\al$ is $t=1$, i.e., we have $\inf_{0\le t\le 1-\eps}\al(t)>0$ for
each $\eps>0$.
\eee
We treat both cases in similar way, yet with slightly different methods.

\subsection{The critical point $t=0$}

We begin with a preliminary result which is interesting in its own right.
\begin{prop}
\label{p1}
If $1<p\le\infty$, then there is a constant $c>0$ only depending on $p$ such that  for
each $0<r\le 1$ and each measurable non-negative $\al(\cdot)$ on $[0,r]$ it
follows that
\be  \label{Rar1}
  \|\Ra : L_p[0,r]\to L_p[0,r]\|\le c\,\sup_{0< t\le r} (2 t)^{\al(t)}\,.
\ee
\end{prop}

\begin{proof}
Let us start with the case $p=\infty$. Here we have
\beaa
\|\Ra f\|_\infty&\le& \|f\|_\infty\,\sup_{0\le t\le r}\frac{1}{\Gamma(\al(t))}\,\int_0^t (t-s)^{\al(t)-1}\d s\\
&\le&
\frac{\|f\|_\infty}{K_0}\sup_{0\le t \le r} t^{\al(t)}
\eeaa
which proves \eqref{Rar1} in that case.
\medskip

Suppose now $1<p<\infty$.
In a first step we assume that $r=2^{-N}$ for some integer $N\ge 0$.
Define intervals $I_n\subseteq[0,1]$ by
$$
   I_n :=\left[2^{-(n+1)},2^{-n}\right]\,,\quad n=0,1,\ldots
$$
and denote by $P_n$ the projections onto $L_p(I_n)$, i.e., we have
$P_n f =f\cdot\ind_{I_n}$. Then $\Ra$ on $L_p[0,2^{-N}]$ admits the representation
\begin{equation} \label{a0}
   \Ra =\sum_{m=0}^\infty \left[\sum_{n=N}^\infty P_n\circ \Ra\circ P_{n+m}\right]
   :=\sum_{m=0}^\infty R_m^{\al(\cdot)}
\end{equation}
where the operators $R_m^{\al(\cdot)}$ are those in the brackets.
In particular, if $m\ge 1$, then for $t\in I_n$ we have
$$
    (R_m^{\al(\cdot)} f)(t)
    =\frac{1}{\Gamma(\al(t))}\int_{I_{n+m}}(t-s)^{\al(t)-1} f(s)\, \d s\,.
$$
Suppose now $m\ge 2$. Then, if $t\in I_n$ and $s\in I_{n+m}$ we get
$$
    2^{-n-2}\le t-s\le 2^{-n}\,,
$$
which implies for those $t$ and $s$ that
$$
     (t-s)^{\al(t)-1}\le 4\cdot 2^{-n(\al(t)-1)} \le 4\cdot 2^{-n(a_n-1)}
$$
where we set $a_n:=\inf_{t\in I_n}\al(t)$.

Consequently, if $t\in I_n$, then by H\"older's inequality we conclude that
\begin{eqnarray} \label{a1}
\nonumber
   |(R_m^{\al(\cdot)} f)(t)|&\le& \frac{4}{K_0} \,2^{-n(a_n-1)}\,\int_{I_{n+m}}|f(s)|\d s
   \\
   &\le& \nonumber
   \frac{4\cdot2^{-1/p'}}{K_0} \,2^{-n(a_n-1)}\,2^{-(n+m)/p'}
   \left(\int_{I_{n+m}}|f(s)|^p\d s\right)^{1/p}
   \\
   &=& c_1\,2^{-n(a_n-1)}\,2^{-(n+m)/p'}\left(\int_{I_{n+m}}|f(s)|^p\d s\right)^{1/p}
\end{eqnarray}
with $c_1:=(2^{-1/p'}\,4)/K_0$.
As usual, $p'$ is defined by $\frac 1 p + \frac 1 {p'}=1$ and
$K_0$ is as in \eqref{defK}. Setting $c_2:=c_1\,2^{-1/p}= 2/K_0$,
estimate (\ref{a1}) yields
\begin{eqnarray*}
    \int_{I_n}|(R_m^{\al(\cdot)} f)(t)|^p\d t &\le& c_1^p\, 2^{-n-1}\,2^{-np(a_n-1)}\,
    2^{-n p/p'}\,2^{-m p/p'} \int_{I_{n+m}}|f(s)|^p\,\d s
    \\
    &=& c_2^p\,2^{- n p\, a_n}\, 2^{-m p/p'} \int_{I_{n+m}}|f(s)|^p\,\d s
\end{eqnarray*}
where we used $ -n +n p - n p/p'=0$. Summing up, for any $f\in L_p[0,2^{-N}]$ holds
\begin{eqnarray*}
    \int_0^{2^{-N}}|(R_m^{\al(\cdot)} f)(t)|^p\d t
    &\le& c_2^p\,2^{-m p/p'}\,\sum_{n\ge N}\,2^{- n p\, a_n}\,
    \int_{I_{n+m}}|f(s)|^p\,\d s
    \\
    &\le& c_2^p\, 2^{-m p/p'}\,\sup_{n\ge N} 2^{- n p\, a_n}\,
    \sum_{n\ge N}\int_{I_{n+m}}|f(s)|^p\,\d s
    \\
    &\le& c_2^p\, 2^{-m p/p'}\,\sup_{n\ge N} 2^{- n p\, a_n}\,\|f\|_p^p\;.
\end{eqnarray*}
In different words, for any $m\ge 2$ we have
\begin{equation} \label{b1}
    \|R_m^{\al(\cdot)}\|\le c_2\, 2^{-m /p'}\,\sup_{n\ge N} 2^{- n  a_n}\,.
\end{equation}
Because of $p>1$, hence $p'<\infty$,  estimate (\ref{b1}) implies
\begin{equation} \label{a2}
    \sum_{m=2}^\infty \|R_m^{\al(\cdot)}\|
    \le c_3 \left[\sup_{n\ge N} 2^{-na_n}\right]
\end{equation}
where $c_3 := c_2\,\sum_{m=2}^\infty 2^{-m/p'}$ .
\bigskip

In view of (\ref{a0}) it remains to estimate $\|R_0^{\al(\cdot)} + R_1^{\al(\cdot)}\|$
suitably. Note that
$$
   R_0^{\al(\cdot)} +R_1^{\al(\cdot)}
   = \sum_{n=N}^\infty P_n\circ \Ra\circ(P_n+P_{n+1})\,,
$$
hence we get
\begin{eqnarray*}
   \|(R_0^{\al(\cdot)} + R_1^{\al(\cdot)}) f\|_p^p
   &=& \sum_{n=N}^\infty \int_{I_n}|\Ra(P_n+P_{n+1}) f(t)|^p\,\d t
   \\
   &=&\sum_{n=N}^\infty \int_{I_n}|R^{\al_n(\cdot)}(P_n+P_{n+1}) f(t)|^p\,\d t
\end{eqnarray*}
where $\al_n(\cdot)$ is the function with
$$
  \al_n(t):=\left\{
     \begin{array}{ccc}
      \al(t) &:& t\in I_n, \\
      a_n    &:& t\notin I_n.
     \end{array}
  \right.
$$
Thus
\begin{eqnarray} \label{a3}
  \nonumber
   \|(R_0^{\al(\cdot)} + R_1^{\al(\cdot)}) f\|_p^p
   &\le&
   \nonumber
   \sum_{n=N}^\infty \|R^{\al_n(\cdot)} : L_p(I_n\cup I_{n+1})\to  L_p(I_n\cup I_{n+1})\|^p
   \,\int_{I_n\cup I_{n+1}}|f(s)|^p\,\d s
   \\
   &\le& 2\,\sup_{n\ge N}
   \|R^{\al_n(\cdot)} : L_p(I_n\cup I_{n+1})\to  L_p(I_n\cup I_{n+1})\|^p\cdot \|f\|_p^p\,.
\end{eqnarray}
Next we want to apply \eqref{sc2} to the interval $I_n\cup I_{n+1}$ and to the operator
$R^{\al_n(\cdot)}$ . To do so we observe that $|I_n\cup I_{n+1}|=3\cdot 2^{-n-2}$ and
that by the definition of $\al_n(\cdot)$ it follows that
$\inf_{t\in I_n\cup I_{n+1}} \al_n(t) = a_n$. Hence \eqref{sc2} gives
$$
   \|R^{\al_n(\cdot)} : L_p(I_n\cup I_{n+1})\to  L_p(I_n\cup I_{n+1})\|
   \le c_p\, (3\cdot 2^{-n-2})^{a_n}
   \le c_p\, 2^{-n\,a_n}\,.
$$
Plugging this into \eqref{a3} implies
\begin{equation} \label{a4}
    \|R_0^{\al(\cdot)}+R_1^{\al(\cdot)}\|\le c_4\sup_{n\ge N} 2^{-n a_n}
\end{equation}
with $c_4 := 2^{1/p}\,c_p$. Combining (\ref{a4}) with (\ref{a2}) we arrive at
\begin{equation} \label{b2}
      \|\Ra : L_p[0,2^{-N}]\to L_p[0,2^{-N}]\|
      \le \|R_0^{\al(\cdot)}+R_1^{\al(\cdot)}\|
      +\sum_{m=2}^\infty\|R_m^{\alpha(\cdot)}\|\le c_5\,\sup_{n\ge N} 2^{-n\,a_n}
\end{equation}
with $c_5:= c_3+c_4$.

In a second step we treat the general case, namely, that $0<r\le 1$ is arbitrary.
Choose a number $N\ge 0$ with $2^{-N-1}\le r\le 2^{-N}$ and extend $\al$ to
$[0,2^{-N}]$ by setting $\al(t):=\al(r)$ whenever $r\le t\le 2^{-N}$. Clearly,
by (\ref{b2}) we have
\begin{eqnarray}   \label{b3}
   \nonumber
   \|\Ra : L_p[0,r]\to L_p[0,r]\|
   &\le& \|\Ra : L_p[0,2^{-N}]\to L_p[0,2^{-N}]\|
   \\
   &\le&  c_5\,\sup_{n\ge N} 2^{-n\,a_n}
\end{eqnarray}
where as before $a_n=\inf\{\al(t) : 2^{-(n+1)}\le t\le 2^{-n}\}$ .
For each $n\ge N$ we find $t_n\ge 2^{-n-1}$ such that
$$
    2^{-n\,a_n}\le 2 \cdot 2^{-n\,\al(t_n)}\;.
$$
Note that we may always choose the $t_n$ in $[0,r]$ by the way of extending
$\al$ to $[0,2^{-N}]$. Clearly this implies
$$
    \sup_{n\ge N} 2^{-n\,a_n}
    \le 2\cdot\sup_{n\ge N}(2^{-n})^{\al(t_n)}
    \le 2\cdot \sup_{n\ge N} (2 t_n)^{\al(t_n)}
    \le 2\cdot\sup_{t\le r} (2 t)^{\al(t)}\,.
$$
The previous estimate combined with (\ref{b3}) leads finally to
$$
    \|\Ra : L_p[0,r]\to L_p[0,r]\|\le c\,\sup_{t\le r} (2 t)^{\al(t)}
$$
with $c:=2\,c_5$.
This completes the proof of the proposition.
\end{proof}

\begin{remark}
One should compare estimate \eqref{Rar1} with that given in \eqref{Rar}. Estimate
\eqref{Rar} makes only sense for $\inf_{0\le t\le r}\al(t)=\al_0>0$ which we do
not suppose in \eqref{Rar1}. On the other hand, if $\al_0>0$, then \eqref{Rar1}
implies \eqref{Rar}, but only for $0<r\le 1/2$ .
\end{remark}
\bigskip

We are now ready to state the main result of this section.

\begin{thm} \label{t1}
   Let $\al$ be a measurable function on $(0,1]$ with
   $\inf_{\varepsilon\le t\le 1}\al(t)>0$ for each $\varepsilon>0$.
   Suppose $1<p\le \infty$. If
\begin{equation} \label{cond1}
   \lim_{t\to 0} t^{\al(t)}=0\,,
\end{equation}
then $\Ra$ is a compact operator in $L_p[0,1]$.
Conversely, if  we have
\begin{equation} \label{cond2}
   \liminf_{t\to 0} t^{\al(t)}>0\,,
\end{equation}
then $\Ra$ is non-compact.
\end{thm}

Before proving Theorem \ref{t1}, let us rewrite it slightly. To this
end, define the function $\vp$ by
\be \label{al}
    \vp(t):=\al(t)\cdot|\ln t|\,,\quad\mbox{i.e.}\quad \al(t)
    =\frac{\vp(t)}{|\ln t|}\,,\quad 0<t\le\eps\,,
\ee
for a sufficiently small $\eps>0$.
Then the following holds.

\begin{thm}
If $\al$ is as in Theorem $\ref{t1}$, then with $\vp$ defined by $(\ref{al})$
the following implications are valid.
\beaa
   \lim_{t\to 0}\vp(t)=\infty\quad &\Longrightarrow&\quad \Ra\;
   \mbox{is a compact operator in}\; L_p[0,1]\,.
   \\
   \limsup_{t\to 0}\vp(t)<\infty\quad &\Longrightarrow&\quad \Ra\;
   \mbox{is a non-compact operator in}\; L_p[0,1]\,.
\eeaa
\end{thm}

\noindent
\textbf{Proof of Theorem \ref{t1} :}
Let us first assume that condition (\ref{cond1}) is satisfied. Fix $r>0$ and
split $\Ra$ as
$$
   \Ra = P_{[0,r]}\circ \Ra + P_{[r,1]}\circ \Ra\,.
$$
Here
$$
   P_{[0,r]} f := f\cdot \ind_{[0,r]}\quad\mbox{while}
   \quad P_{[r,1]} f := f\cdot \ind_{[r,1]}\,.
$$
Note that $P_{[r,1]}\circ \Ra$ is compact from $L_p[0,1]$ into $L_p[0,1]$.
Indeed, if we define $\tilde\al$ by
$$
  \tilde\al (t):=
  \left\{
     \begin{array}{ccc}
        \al(r)&:& 0\le t\le r\\
        \al(t)&:& r\le t\le 1
     \end{array}
  \right.
$$
then it follows that
$$
    P_{[1,r]}\circ R^{\tilde\al(\cdot)} = P_{[1,r]}\circ \Ra\,.
$$
By the properties of $\al$ we have
$$
    \inf_{0\le t\le 1}\tilde\al(t)>0\,,
$$
thus Proposition \ref{comp} applies and $R^{\al_r(\cdot)}$ is compact,
hence so is $P_{[r,1]}\circ \Ra$.

Observe that $P_{[0,r]}\circ \Ra$ is nothing else as $\Ra$ regarded as operator
from $L_p[0,r]$ into $L_p[0,r]$. Consequently, Proposition \ref{p1} applies and
leads to
\be \label{b4}
    \|P_{[0,r]}\circ \Ra\|\le c\,\sup_{0\le t\le r} (2t)^{\al(t)}\,.
\ee
We claim now that (\ref{cond1}) yields $\lim_{t\to 0} (2 t)^{\al(t)}=0$.
To see this, write
$$
    (2 t)^{\al(t)}=(2\sqrt t)^{\al(t)}\left[t^{\al(t)}\right]^{1/2}
    \le \left[t^{\al(t)}\right]^{1/2}
$$
provided that $t\le 1/4$, hence
$$
    \limsup_{t\to 0} (2 t)^{\al(t)}
    \le \left[\limsup_{t\to 0} t^{\al(t)}\right]^{1/2}\,,
$$
and by (\ref{b4}) condition (\ref{cond1}) leads to
$$
     \lim_{r\to 0}  \|P_{[0,r]}\circ \Ra\| = 0\,.
$$
Thus, as $r\to 0$, the operator $\Ra$ is a limit (w.r.t.~the operator norm) of
the compact operators $P_{[r,1]}\circ \Ra$, hence it is compact as well.
This proves the first part of the theorem.
\bigskip

To verify the second part, we first prove a preliminary result.
Let as above $I_n=[2^{-(n+1)},2^{-n}]$ and set
$$
   b_n:= \sup_{t\in I_n} \al(t)\,,\quad n=0,1,\ldots
$$
We start with showing the following:
Suppose that
\begin{equation} \label{cond3}
   \inf_{n\ge 0}2^{-n\,b_n}>0\,,
\end{equation}
then $\Ra$ is non-compact. To verify this, set
$$
    h_n:= 2^{(n+1)/p}\,\ind_{I_n}\,,\quad n=0,1,2,\ldots.
$$
Then $\|h_n\|_p=1$ and for $t\in I_n$ we have
\begin{eqnarray*}
    \Ra h_n (t)&=& \frac{2^{(n+1)/p}}{\Gamma(\al(t))}\,
    \int_{2^{-n-1}}^t (t-s)^{\al(t)-1} \d s
    \\
    &\ge& c_1\,2^{n/p}\,(t-2^{-n-1})^{\al(t)}
    \ge c_1\,2^{n/p}\,(t-2^{-n-1})^{b_n}
\end{eqnarray*}
where $c_1:=\frac{2^{1/p}}{K_0}$.
From this we derive
\begin{eqnarray} \label{e1}
  \nonumber
  \|(\Ra h_n)\ind_{I_n}\|_p^p  &\ge& c_1^p\,2^n\,
  \int_{3\cdot 2^{-n-2}}^{2^{-n}} (t-2^{-n-1})^{b_n p}\,\d t
  \\
  &\ge& c_1^p\,2^n\, 2^{-n-2}\, 2^{-(n+2)\,b_n p}\,.
\end{eqnarray}
Using
$$
    2^{-(n+2)\,b_n}=\left[2^{-n\,b_n}\right]^{(n+2)/n}\,,
$$
we see that assumption (\ref{cond3}) and estimate (\ref{e1}) lead to
$$
   \liminf_{n\to\infty}\|(\Ra h_n)\ind_{I_n}\|_p >0\,.
$$
But this implies that $\Ra$ is non-compact. Indeed, if $m<n$, then
$(\Ra h_m)(t)=0$ for $t\in I_n$, hence for some $\delta>0$ we have
$$
   \|\Ra h_n -\Ra h_m\|_p \ge \|(\Ra h_n)\ind_{I_n}\|_p\ge \delta
$$
provided that $m$ is sufficiently large.
Thus there are infinitely many functions in the unit ball of $L_p[0,1]$ such
that the mutual distance between their images is larger than $\delta >0$.
Of course, an operator possessing this property cannot be compact.

To complete the proof it suffices to verify that (\ref{cond2}) implies
(\ref{cond3}). Choose $t_n\le 2^{-n}$ for which $\al(t_n)$ almost attains
$b_n$, i.e., for which
$$
   2^{-n\al(t_n)} \le 2\cdot 2^{-n\,b_n}\,,\quad n=0,1,\ldots
$$
Then we get
\be
\label{es1}
   2^{-n\,b_n}  \ge 2^{-1}\cdot 2^{-n\al(t_n)}
   \ge 2^{-1}\cdot t_n^{\al(t_n)}
   \ge 2^{-1}\cdot \inf_{0<t\le 1} t^{\al(t)}\,.
\ee
By the assumptions about $\al(\cdot)$ for any $\eps>0$ we have
$\inf_{\eps\le t\le 1} t^{\al(t)}>0$,
hence because of \eqref{es1} condition (\ref{cond2}) implies (\ref{cond3}) and this completes the proof.
\hspace*{\fill}$\square$\medskip\par
\bigskip

\begin{remark}
Note that there is only one very special case not covered by Theorem \ref{t1}.
Namely, if we have
$\lim_{t\to 0}\al(t)=0$, $\liminf_{t\to 0}\al(t)|\ln t| <\infty$ and
$\limsup_{t\to 0}\al(t)|\ln t| =\infty$.
\end{remark}

\subsection{The critical point $t=1$}

We suppose now that $\inf_{0\le t\le 1-\eps}\al(t)>0$ for each $\eps>0$.
One might expect that this case can be transformed into the first one,
i.e., in the case that the critical point is $t=0$, by an easy time inversion.
But if $S: L_p[0,1]\to L_p[0,1]$ is defined by
$$
   (S f)(t):= f(1-t)\,,\quad 0\le t\le 1\,,
$$
then we get
\be \label{S}
    (S\,\Ra S f)(t)
    = \frac{1}{\Gamma(\tilde\al(t))}\int_t^1 (s-t)^{\tilde\al(t)-1} f(s)\,\d s
\ee
where $\tilde\al(t):=\al(1-t)$.
The problem is that the right hand expression is \textit{not} $R^{\tilde\al(\cdot)} f$.
Thus, although a time inversion transforms the critical point $t=1$ of $\al$ into
the critical point $t=0$ of $\tilde\al$, it does not solve our problem because the
inversion changes the fractional integral as well. It is also noteworthy to mention
that the operator in (\ref{S}) is \textit{not} the dual operator of
$R^{\tilde\al(\cdot)}$, hence also a duality argument does not apply here.

Therefore, as far as we see, a time inversion is not useful to investigate the critical
case $t=1$, thus we are forced to adapt our former methods to the new situation.
\medskip

We start with a proposition which is the counterpart of Proposition \ref{p1} in that
case. Its proof is similar to that of Proposition \ref{p1}, yet the arguments differ
at some crucial points.

\begin{prop} \label{p2}
There is a constant $c>0$ only depending on $p>1$ such that
for any real $0<r<1/2$ it follows
$$
   \|\Ra : L_p[1-r,1]\to L_p[1-r,1]\|
   \le c\max\{\sup_{0<t\le r} (2t)^{\tilde\al(t)/2}, r^{1/2p}\}
$$
where as before $\tilde \al(t)=\al(1-t)$ for $0\le t\le 1$.
\end{prop}

\begin{proof}
The case $p=\infty$ can be treated exactly as in the proof of Proposition \ref{p1}.

Thus let us assume
$1<p<\infty$.
Again we first suppose that $r=2^{-N}$ for a certain integer $N\ge 1$
and split the interval $[1-2^{-N},1]$ by dyadic intervals $I_n$ which
are this time defined by
\be \label{In}
   I_n:=[1-2^{-n},1-2^{-(n+1)}]\,,\quad n=N,N+1,\ldots
\ee
Then $\Ra$ on $L_p[1-2^{-N},1]$ can be represented as
$$
    \Ra = \sum_{{n\ge N}\atop{k\ge N}}P_n \Ra P_k
    = \sum_{n\ge k\ge N} P_n \Ra P_k = \sum_{m=0}^\infty \Ra_m
$$
where as before $P_n f = f\circ \ind_{I_n}$ and
$$
   \Ra_m := \sum_{n=m+N}^\infty P_n \Ra P_{n-m}.
$$
In particular, if $m\ge 1$ and
$t\in I_n$ with $n\ge N+m$, then it follows that
$$
    (\Ra_m f)(t)
    = \frac{1}{\Gamma(\al(t))}\int_{I_{n-m}} (t-s)^{\al(t)-1} f(s) \d s\,.
$$
Assuming $m\ge 2$ and $n\ge N+m$ for $t\in I_n$ and $s\in I_{n-m}$ one
easily gets
$$
    2^{-n+m-2}\le t-s\le 2^{-n+m}\,,
$$
hence
\be \label{d2}
    (t-s)^{\al(t)-1}\le 4\cdot \left(2^{-n+m}\right)^{\al(t)-1}
\ee
whenever $t$ and $s$ are as before.
With $c_1:=\frac{4}{K_0}$ estimate (\ref{d2})
leads for $t\in I_n$ to
\beaa
   |\Ra_m f (t)|
   &\le& c_1\left(2^{-n+m}\right)^{\al(t)-1}\,\int_{I_{n-m}}|f(s)|\d s
   \\
   &\le& c_1\left(2^{-n+m}\right)^{\al(t)-1} |I_{n-m}|^{1/p'}
   \left(\int_{I_{n-m}}|f(s)|^p\d s\right)^{1/p}
   \\
   &\le& c_2\, (2^{-n+m})^{a_n-1} (2^{-n+m})^{1/p'}
   \left(\int_{I_{n-m}}|f(s)|^p\d s\right)^{1/p}
\eeaa
where as above $a_n:=\inf_{t\in I_n}\al(t)$ and $c_2:=2^{-1/p'}c_1$. Consequently,
whenever $n\ge N+m$ and $m\ge 2$, with $c_3:=2^{-1/p} c_2=2/K_0$ this implies
\beaa
    \int_{I_n}|\Ra_m f)(t)|^p\d t &\le& c_3^p \,2^{-n}\,
    (2^{-n+m})^{p/p'} (2^{-n+m})^{p a_n-p}\,\int_{I_{n-m}}|f(s)|^p\d s
    \\
    &=& c_3^p\, 2^{-m +m  p a_n -n p a_n}\,\int_{I_{n-m}}|f(s)|^p\d s
\eeaa
because of $-n -n p/p' +n p= 0$ and $m p/p'  - m p = -m$. Summing the last
estimate over all $n\ge N+m$ we arrive at
\be \label{d3}
   \|\Ra_m\|\le c_3\, \sup_{n\ge m+N}\,2^{-m/p}\,2^{(-n+m) a_n}\,.
\ee
If $a_n\le 1/2p$, then it follows that
$$
   2^{-m/p}\,2^{(-n+m) a_n}\le 2^{-m/2p}\, 2^{- n a_n}
$$
while for $a_n\ge 1/2p$ and $n\ge m+N$ we get
$$
2^{-m/p}\,2^{(-n+m) a_n}\le 2^{-m/p} \, 2^{-N/2p}\,.
$$
Thus (\ref{d3}) finally leads to
\be \label{d4}
     \sum_{m=2}^\infty \|\Ra_m\|
     \le c_4\, \max\{\sup_{n\ge N} 2^{- n a_n}, 2^{-N/2p}\}
\ee
with $c_4:= c_3\,\sum_{m=2}^\infty 2^{-m/2p}$.

Our next aim is to estimate $\|\Ra_0 +\Ra_1\|$ suitably. Note that
\beaa
    \Ra_0 +\Ra_1&=& P_N\,\Ra\, P_N + \sum_{n=N+1}^\infty P_n\, \Ra\, (P_n+ P_{n-1})
    \\
    &=& P_N\,\Ra\, P_N + \sum_{n=N+1}^\infty P_n\, R^{\al_n(\cdot)}\, (P_n+ P_{n-1})
\eeaa
where $\al_n(t)= \al(t)$ for $t\in I_n$ and $\al_n(t)=a_n$ whenever $t\notin I_n$.
For $f\in L_p[1-2^{-N},1]$ this implies
\bea
\label{d5} \nonumber
   \lefteqn{\|(\Ra_0+\Ra_1) f\|_p^p
   \le \|\Ra : L_p(I_N)\to L_p(I_N)\|^p \int_{I_N} |f(s)|^p\d s}
   \\
   \nonumber
   &+& \sum_{n=N+1}^\infty\|R^{\al_n(\cdot)}:L_p(I_n\cup I_{n-1})
   \to L_p(I_n\cup I_{n-1})\|^p\int_{I_n\cup I_{n-1}}|f(s)|^p\d s
   \\
   \nonumber
   &\le& \|\Ra : L_p(I_N)\to L_p(I_N)\|^p\,\|f\|_p^p\\
   &+& 2\sup_{n\ge N+1}
   \|R^{\al_n(\cdot)}:L_p(I_n\cup I_{n-1})\to L_p(I_n\cup I_{n-1})\|^p
   \|f\|_p^p\,.
\eea
Since $|I_N|= 2^{-N-1}\le 2^{-N}$ and
$|I_n\cup I_{n-1}| =  3 \cdot 2^{-n-1}\le 2^{-n+1}$, exactly as in Proposition \ref{p1} estimates \eqref{sc2} and \eqref{d5}
imply
\be
\label{es2}
\|\Ra_0+\Ra_1\|\le c_5 \Big[2^{-N a_N} + \sup_{n\ge N+1}\,\left(2^{-n+1}\right)^{a_n}\Big]
\ee
with
$c_5:=2^{1/p}\,c_p$ and $c_p$ is the constant in \eqref{sc2}.
Recall that $N\ge 1$,
hence the numbers $n$ in the supremum of the right hand side of \eqref{es2} satisfy $n\ge 2$ and we have $\frac{n-1}{n}\ge \frac 1 2$.
Thus \eqref{es2} leads to
\be \label{d6}
    \|\Ra_0+\Ra_1\|\le c_5\Big[ 2^{-N a_N} +\,\sup_{n\ge N+1}\,\left(2^{-n a_n}\right)^{(n-1)/n}\Big]
    \le 2\,c_5\,\sup_{n\ge N}\,2^{-n a_n/2}
\ee
Combining (\ref{d4}) with (\ref{d6}) gives
$$
     \|\Ra : L_p[1-2^{-N},1]\to  L_p[1-2^{-N},1]\|
     \le c_6\, \max\{\sup_{n\ge N} 2^{- n a_n/2}, 2^{-N/2p}\}\,.
$$
where $c_6:=\max\{c_4,2\,c_5\}$.

For arbitrary $r\in(0,1/2]$ choose an integer $N\ge 1$ with $2^{-N-1}\le r \le 2^{-N}$
and extend $\al$ to $[1-2^{-N},1]$ by setting $\al(t):= \al(1-r)$ whenever
$1-2^{-N}\le t< 1-r$.
Then we conclude
\bea
   \nonumber
   \label{d8}
   \|\Ra : L_p[1-r,1]\to  L_p[1-r,1]\|
   &\le& \|\Ra : L_p[1-2^{-N},1]\to  L_p[1-2^{-N},1]\|\
   \\
   &\le& c_6\, \max\{\sup_{n\ge N} 2^{- n a_n/2}, 2^{-N/2p}\}\,.
\eea
Furthermore, we choose $t_n\in[ 2^{-n-1}, 2^{-n}]$ so
that $\al(1-t_n)$ "almost" attains the infimum $a_n$ of $\al(\cdot)$ on $I_n$,
i.e., that we have
$$
    2^{-n a_n}\le 2\cdot 2^{-n \al(1-t_n)}\,.
$$
Hence,
$$
    \sup_{n\ge N} 2^{-n a_n/2}
    \le \sqrt 2\cdot\sup_{n\ge N} 2^{-n\al(1-t_n)/2}
    \le \sqrt 2\cdot \sup_{n\ge N}(2 t_n)^{ n\al(1-t_n)/2}
    \le \sup_{0<t\le r} (2t)^{n \tilde\al(t)/2}\,.
$$
and
$$
    2^{-N/2p}\le (2 r)^{1/2p}= 2^{1/2p}\cdot r^{1/2p} \,,
$$
thus (\ref{d8}) completes the proof by changing $c_6$ to $c:= 2^{1/2p}\,c_6$.
\end{proof}
\medskip

Of course, Proposition \ref{p2} may also be formulated as follows:

\begin{prop}
For each $1/2\le \theta <1$ we have
\be \label{d8a}
    \|\Ra : L_p[\theta,1]\to L_p[\theta,1]\|
    \le c\,\max\{\sup_{\theta\le t <1} (2(1-t))^{\al(t)/2}, (1-\theta)^{1/2p}\}
\ee
with $c>0$ only depending on $p>1$.
\end{prop}

\begin{cor}
For  $\al(\cdot)$ on $[0,r]$ define $Q^{\al(\cdot)}$
on $L_p[0,r]$ by
$$
   (Q^{\al(\cdot)} f)(t)
   := \frac{1}{\Gamma(\al(t))}\int_t^r (s-t)^{\al(t)-1}\,f(s) \d s\,,
   \quad 0\le t\le r\,.
$$
If $0<r\le 1/2$, then it follows that
$$
    \|Q^{\al(\cdot)} : L_p[0,r]\to L_p[0,r]\|
    \le c\,\max\{\sup_{0<t\le r} (2 t)^{\al(t)/2}, r^{1/2p}\}\,.
$$
\end{cor}

\begin{proof}
The proof is an immediate consequence of Proposition \ref{p2} and representation
\eqref{S} which now may be written as
$$
    S\circ R^{\tilde\al(\cdot)} \circ S = Q^{\al(\cdot)}
$$
where $\tilde\al(t):=\al(1-t)$, while $R^{\tilde\al(\cdot)}$ is defined
on $L_p[1-r,1]$ and $Q^{\al(\cdot)}$ on $L_p[0,r]$.
\end{proof}

\bigskip

The next result is the counterpart to Theorem \ref{t1}.

\begin{thm}
\label{t2}
Let $\al$ be a measurable function on $[0,1)$ so that
$\inf_{0\le t\le \theta} \al(t)>0$ for each $\theta<1$.
Suppose $1<p\le \infty$. If
\be \label{d9}
     \lim_{t\to 1} (1-t)^{\al(t)}=0\,,
\ee
then $\Ra$ is a compact operator in $L_p[0,1]$. Conversely, if
\be \label{d10}
    \liminf_{t\to 1} (1-t)^{\al(t)}>0\,,
\ee
then $\Ra$ is non-compact.
\end{thm}

\begin{proof}
For a given $1/2\le \theta<1$ we decompose $\Ra$ as
$$
\Ra = P_{[0,\theta]}\,\Ra + P_{[\theta,1]}\,\Ra\,.
$$
Now we proceed exactly as in the proof of Theorem \ref{t1}. The operator
$P_{[0,\theta]}\,\Ra$ is compact and as before assumption (\ref{d9}) implies
$$
    \lim_{\theta\to 1}
    \max\{\sup_{\theta\le t <1} (2(1-t))^{\al(t)/2}, (1-\theta)^{1/2p}\} = 0\,.
$$
In view of (\ref{d8a}) it follows
$\lim_{\theta\to 1}\|P_{[\theta,1]}\,\Ra\|=0$, hence, if $\theta\to 1$,
the operator $\Ra$ is approximated by the compact operators
$P_{[0,\theta]}\,\Ra$, consequently $\Ra$ is compact as well.

The second part of the theorem is also proved by similar methods as in
Theorem \ref{t1}, yet with a small change. With the intervals $I_n$ in
\eqref{In} we define functions $h_n$ by
$$
   h_n = 2^{(n+1)/p}\,\ind_{I_n}\,,\quad n=0,1,2,\ldots
$$
As in the proof of Theorem \ref{t1} condition (\ref{d10}) implies that
for some $\delta>0$ and $n$ sufficiently large
$\|(\Ra h_n)\ind_{I_n}\|_p\ge \delta$. If $m<n$, then this time the interval
$I_m$ is on the left hand side of $I_n$, so we get $(\Ra h_n)(t)=0$ whenever
$t\in I_m$. Hence,
$$
   \|\Ra h_m -\Ra h_n\|_p \ge \|(\Ra h_m)\ind_{I_m}\|_p \ge \delta
$$
provided that $m_0\le m<n$ for a certain $m_0\in\N$. Thus the operator
$\Ra$ is non-compact as claimed.
\end{proof}

\section{Entropy estimates for classical Riemann--Liouville operators}
\label{Eec}
\setcounter{equation}{0}

Given a compact operator $S$ between two Banach spaces $E$ and $F$,
its degree of compactness is mostly measured by the behavior of its
entropy numbers $e_n(S)$. Let us shortly recall the definition of these numbers.
\begin{defn}
Let $[E,\|\cdot\|_E]$ and $[F,\|\cdot\|_F]$ be Banach spaces with unit balls
$B_E$ and $B_F$, respectively. Given a (bounded) operator $S$ from $E$ into $F$,
its $n$-th (dyadic) entropy number $e_n(S)$ is defined by
$$
   e_n(S):=\inf\left\{\eps>0\, : \exists\, y_1,\ldots,y_{2^{n-1}}\in F\;
  \ \emph{such that}\quad S(B_E)\subseteq\bigcup_{j=1}^{2^{n-1}} (y_j+\eps B_F)\right\}\;.
$$
\end{defn}
Note that $S$ is a compact operator if and only if $\lim_{n\to\infty} e_n(S)=0$.
Furthermore, the faster $e_n(S)$ tends to zero as $n\to\infty$, the higher
(or better) is the degree of compactness of $S$. We refer to \cite{CS} or to
\cite{ET} for further properties of entropy numbers.
\medskip

Our final aim is to find suitable estimates for $e_n(\Ra)$ in dependence
of properties of the function $\al(\cdot)$.
But before we will be able to do this, we need some very precise estimates
for $e_n(R^\al)$ in the classical case.
We start with citing what is known about the entropy behavior for those
operators. For an implicit proof in the language of embeddings we refer
to \cite{ET}, 3.3.2 and 3.3.3; a rigorous one was recently given in
\cite{CHR}, Theorem 5.21. For special $p$ and $q$ the result was also
proved by other authors, for example in \cite{KL}, \cite{LL} or \cite{LS}.

\begin{prop}
\label{enr}
Suppose $1\le p,q\le \infty$ and $\al > (1/p-1/q)_+$. Then there are positive
constants $c_{\al,p,q}$ and $C_{\al,p,q}$ such that
\be
\label{enRa}
   c_{\al,p,q}\,n^{-\al}
   \le e_n(R^\al : L_p[0,1]\to L_q[0,1])
   \le C_{\al,p,q}\,n^{-\al}\,.
\ee
\end{prop}
\medskip

The main objective of this section is to improve the right hand estimate
in \eqref{enRa} as follows:

\begin{thm}
\label{te1}
Suppose $1\le p,q\le \infty$.
\bee
\item
If $1\le q\le p\le\infty$, then for each real $b>0$ there is a constant
$c_b>0$ independent of $p$ and $q$ such that for all
$n\ge 1$ and all $\al\in (0,b]$ we have
\be
   \label{enRa1}
   e_n(R^\al : L_p[0,1]\to L_q[0,1])\le c_b\,n^{-\al}\,.
\ee
\item
Suppose $1\le p<q\le \infty$. Then for all $a>\frac 1 p -\frac 1 q$ and
$b>a$ there is a constant $c_{a,b}>0$ (maybe depending on $p$ and $q$)
such that for $n\ge 1$ and $a\le \al\le b$ it follows that
\be
  \label{enRa2}
  e_n(R^\al : L_p[0,1]\to L_q[0,1])\le c_{a,b}\,n^{-\al}\,.
\ee
\eee
\end{thm}
\begin{remark}
We do not know whether estimate \eqref{enRa2} even holds with a constant
$c_b>0$ only depending on $b$ and for all $\frac 1 p -\frac 1 q <\al\le b$.
\end{remark}
\medskip

The proof of Theorem \ref{te1} needs some preparation. We start with
introducing the necessary notation. A basic role in the proof is played by
approximation numbers defined as follows:

\begin{defn}
Given Banach spaces $E$ and $F$ and an operator $S$ from $E$ into $F$,
its $n$--th approximation number $a_n(S)$ is defined by
$$
    a_n(S):= \inf\{ \|S-A\| : A : E\to F\:\mathrm{and}\;
    \mathrm{rank}(A)<n\}\,.
$$
\end{defn}

For the main properties of approximation numbers we refer to \cite{Pie1}
and to \cite{Pie2}.
\medskip

A second basic ingredient in the proof of Theorem \ref{te1} are some special
Besov spaces. To introduce them we need the following definition:
Given $f\in L_p[0,1]$ and $0 < h\le 1$, we define the function $\Delta_h f$ by
$$
(\Delta_h f)(t) := f(t+h) -f(t)\,,\quad 0\le t\le 1-h\,.
$$
\begin{defn}
If $0<\al<1$, then the Besov space $B_{p,\infty}^\al$  consists of all
functions $f\in L_p[0,1]$ (continuous $f$ if $p=\infty$) for which
\be
\label{normB}
\|f\|_{p,\al}:= \|f\|_p + \sup_{0<h\le 1} h^{-\al}\,\|\Delta_h f\|_p <\infty\,.
\ee
\end{defn}
\medskip

Now we are in position to state and to prove a first important step in the
verification of Theorem \ref{te1}.
\begin{prop}
\label{pr1}
Suppose $0<\al<1$ and
let $I_p : B_{p,\infty}^\al \to L_p[0,1]$ be the natural embedding. Then, if
$1\le p<\infty$, it holds
\be
\label{anI}
a_n(I_p) \le 2^\al\,\left(\frac{2}{1+\al p}\right)^{1/p}\,
n^{-\al}\le 4\,n^{-\al}
\ee
while
\be
\label{anIinf}
a_n(I_\infty)\le 2^\al\,n^{-\al}\le 2\,n^{-\al}\,.
\ee
\end{prop}
\begin{proof}
With
$$
I_j:=\left[\frac{j-1}{n},\frac j n\right]\,,\quad j=1,\ldots,n\,,
$$
we define the operator $P_n : L_p[0,1]\to L_p[0,1]$ as
$$
(P_n f)(t):= \sum_{j=1}^n \int_{I_j} f(s)\,\d s\, \frac{\ind_{I_j}(t)}{|I_j|}
= n\,\sum_{j=1}^n \int_{I_j} f(s)\,\d s\, \ind_{I_j}(t)\,.
$$
Of course, it is true that $\mathrm{rank}(P_n)=n$, hence by the definition of
approximation numbers we get
\be
\label{an1}
a_{n+1}(I_p)\le \|I_p - P_n\, I_p\|=\sup_{\|f\|_{p,\al}\le 1}\|f- P_n f\|_p\,.
\ee
To estimate the right hand side of \eqref{an1} let us first treat the case $1\le p<\infty$. Thus choose any $f\in B_{p,\infty}^\al$ with $\|f\|_{p,\al}\le 1$. That is, for all $0<h\le 1 $ we have
$$
\|f\|_p + h^{-\al}\,\|\Delta_h f\|_p\le 1
$$
yielding in particular
\be
\label{estns}
\|\Delta_h f\|_p\le h^\al
\ee
for all $h$ with $0<h\le 1$.

Now we are in position to estimate the right hand side of \eqref{an1} as follows:
\bea
\label{aPn1}
\nonumber
\|f- P_n f\|_p^p&=& \sum_{j=1}^n \int_{I_j}|f(s)- P_n f(s)|^p\,\d s = n^p\,\sum_{j=1}^n\int_{I_j}\Big|
\int_{I_j}[f(s)-f(t)]\d t\Big|^p\d s\\
&\le&
n^p\,\sum_{j=1}^n\int_{I_j}\Big[\int_{I_j}|f(s)-f(t)|\d t\Big]^p\d s\,.
\eea
An application of H\"older's inequality to the inner integral gives
\be
\label{Hol}
\Big[\int_{I_j}|f(s)-f(t)|\d t\Big]^p\le|I_j|^{p/p'}\cdot \int_{I_j}|f(s)-f(t)|^p\d t
=
n^{-p/p'}\, \int_{I_j}|f(s)-f(t)|^p\d t\,.
\ee
Plugging \eqref{Hol} into \eqref{aPn1} leads to
\bea
\label{aPn2}
\nonumber
\|f-P_n f\|_p^p&\le& n\,\sum_{j=1}^n\int_{I_j}\int_{I_j}|f(s)-f(t)|^p\d s \d t
\le
n\, \int_{\{|t-s|\le 1/n\}}|f(s)-f(t)|^p\,\d s \d t\\
&=& 2\,n\,\int_{\{t\le s\le (t+1/n)\wedge 1\}}|f(s)-f(t)|^p \d s \d t
\eea
because of
$$
\bigcup_{j=1}^n \left(I_j\times I_j\right)\subseteq \Big\{(t,s)\in [0,1]^2 : |t-s|\le \frac 1 n\Big\}\,.
$$
Now \eqref{aPn2} may also be written as
\beaa
2\,n\,\int_0^1 \int_{t}^{(t+1/n)\wedge 1} |f(s)-f(t)|^p\d s \d t&=&
2\,n\,\int_0^{1} \int_{0<h\le (1/n)\wedge (1-t)}|f(t+h) -f(t)|^p\d h \d t\\
&=&
2\,n\,\int_{0<h\le 1/n}\int_0^{1-h} |(\Delta_h f)(t)|^p\d t \d h\\
&\le&
2\,n\, \int_0^{1/n} h^{\al p}\d h =\frac{2}{\al p +1}\, n^{-\al p}
\eeaa
where the estimate in the last line follows by \eqref{estns}. Summing up, we get
$$
\|f- P_n f\|_p \le \left(\frac{2}{\al p +1}\right)^{1/p}\, n^{-\al }
$$
whenever $\|f\|_{p,\al}\le 1$. Hence it follows
$$
a_{n+1}(I_p)\le \left(\frac{2}{\al p +1}\right)^{1/p}\, n^{-\al }\le 2^\al\,
\left(\frac{2}{\al p +1}\right)^{1/p}\, (n+1)^{-\al }\,.
$$
Since $a_1(I_p)=\|I_p\|\le 1$, estimate \eqref{anI} holds for all numbers $n\ge 1$ and this completes the proof
for $p<\infty$.
\medskip

The case $p=\infty$ is even easier. Here we have
$$
|f(t+h)-f(t)|\le h^\al\,,\quad 1\le t\le 1-h\,,
$$
i.e.,
$$
|f(s)-f(t)|\le h^{\al}\,,\quad 0\le t,s\le 1\,, \; |t-s|\le h\,.
$$
Then we get
\beaa
\|f-P_n f\|_\infty &=&\sup_{0\le t\le 1}|f(t) -(P_n f)(t)| \le n\,\sup_{1\le j\le n}\sup_{t\in I_j}\int_{I_j}|f(t)-f(s)|\d s\\
&\le& n \, n^{-\al}\,|I_j|= n^{-\al}\,.
\eeaa
Now we proceed as in the case $p<\infty$ and arrive at
$$
a_n(I_\infty)\le 2^\al\, n^{-\al}
$$
as asserted.
\end{proof}

\begin{remark}
The fact $a_n(I_p)\approx n^{-\al}$ is well--known (cf.~\cite{ET}), yet does not suffice for our purposes. We have to have a uniform upper bound for $a_n(I_p)$ as in \eqref{anI} or \eqref{anIinf}, respectively.
\end{remark}
\medskip

Another basic fact will play a crucial role in the proof of Theorem \ref{te1}. It was recently proved in \cite{CHR} (cf.~Lemma 5.19).
\begin{prop}
If $0<\al<1$, then for each $f\in L_p[0,1]$
we have
$$
\|\Delta_h(R^\al f)\|_p \le \frac{2}{\Gamma(\al+1)}\,h^\al \|f\|_p\le \frac{2}{K_0}\,h^\al\,\|f\|_p\,.
$$
\end{prop}
Since
$$
\|R^\al f\|_p\le \frac{1}{K_0}\,\|f\|_p\,,
$$
by definition \eqref{normB} we get the following result:
\begin{prop}
\label{pr2}
If $0<\al<1$, then
$$
\|R^\al f\|_{p,\al}\le \frac{3}{K_0}\,\|f\|_p\,,\quad f\in L_p[0,1]\,,
$$
i.e., $R^\al$ is a bounded operator from $L_p[0,1]$ into $B_{p,\infty}^\al$ with operator norm $\|R^\al\|\le \frac{3}{K_0}$.
\end{prop}

Now we are ready to prove Theorem \ref{te1}.
\medskip

\noindent
\textbf{Proof of Theorem \ref{te1}:} We start with the case $1\le q\le p\le \infty$. Then $(1/p-1/q)_+=0$, hence for any $\al>0$ the operator $R^\al$ is compact from $L_p[0,1]$ into $L_q[0,1]$. In particular, $R^\al$ maps $L_p[0,1]$ into $L_p[0,1]$, hence, if $I_{p,q}$ denotes the natural embedding from $L_p[0,1]$ to $L_q[0,1]$ it follows that
$$
(R^\al : L_p[0,1]\to L_q[0,1]) = I_{p,q}\circ (R^\al : L_p[0,1]\to L_p[0,1])\,,
$$
which yields
$$
e_n(R^\al : L_p[0,1]\to L_q[0,1])\le \|I_{p,q}\|\cdot e_n(R^\al : L_p[0,1]\to L_p[0,1])=
e_n(R^\al : L_p[0,1]\to L_p[0,1])\,.
$$
Consequently, it suffices to verify \eqref{enRa1} in the case $q=p$.

In a first step we suppose $0<b<1$. Then Propositions \ref{pr1} and \ref{pr2} apply and lead to
$$
a_n(R^\al)\le \|R^\al : L_p[0,1]\to B_{p,\infty}^\al\|\cdot a_n(I_p)\le c_0 \, n^{-\al}
$$
where, for example, $c_0$ may be chosen as $12/K_0$. Now, if $1\le b<2$ and $\al\le b$, we get
$$
a_{2n-1}(R^\al)\le \left[a_n(R^{\al/2})\right]^2\le c_0^2\, n^{-\al}\,,
$$
hence
$$
a_n(R^\al)\le 2^\al\, c_0^2\, n^{-\al}\le 2^b\,c_0^2\,n^{-\al}
$$
provided that $0<\al \le b$. Iterating further for each $b>0$ there is a $C_b>0$ with
\be
\label{anRa}
a_n(R^\al)\le C_b\,n^{-\al}
\ee
with $C_b$ independent of $p$.
\medskip

Next we refer to Theorem 3.1.1 in \cite{CS} which asserts the following:
Let $S$ be an operator between real Banach spaces $E$ and $F$. For each $0<\al<\infty$ there is a constant $C_\al>0$ such that for all $N\ge 1$ it follows that
$$
\sup_{1\le n\le N} n^\al\, e_n(S)\le C_\al\, \sup_{1\le n\le N} n^\al\, a_n(S)\,.
$$
Hereby $C_\al$ may be chosen as $C_\al= 2^7(32(2+\al))^\al$ .

Applying this result together with \eqref{anRa} gives for each $N\ge 1$ that
$$
\sup_{1\le n\le N} n^{\al}\,e_n(R^\al)\le C_\al\, \sup_{1\le n\le N} n^\al\, a_n(R^\al)\le C_\al\,C_b
$$
provided that $0<\al\le b$. Since $C^b:=\sup_{0<\al\le b} C_\al<\infty$, this implies (note that $N\ge 1$ is arbitrary)
$$
e_n(R^\al)\le c_b\,n^{-\al}
$$
with constant $c_b:= C^b\,C_b$ only depending on $b$. This completes the
proof of the first part of Theorem \ref{te1}.
\medskip

We turn now to the proof of \eqref{enRa2}. Here $1/p -1/q >0$, hence $R^\al$
is compact from $L_p[0,1]$ to $L_q[0,1]$ only if $\al > 1/p -1/q$ . Take any
pair $a,b$ of real numbers with $1/p-1/q <a<b<\infty$ and choose
$\al\in (a,b]$ arbitrarily. Then we may decompose $R^\al$ as follows:
\be
\label{dec}
(R^\al : L_p[0,1]\to L_q[0,1])
= (R^a : L_p[0,1]\to L_q[0,1])\circ (R^{\al-a} : L_p[0,1]\to L_p[0,1]).
\ee
Because of $a>1/p -1/q$ the operator  $R^a$ on the right hand side of
\eqref{dec} is compact and by \eqref{enRa} we have
\be
\label{r1}
e_n(R^a)\le C_{a,p,q}\, n^{-a}\,.
\ee
On the other hand, $R^{\al-a}$ maps $L_p[0,1]$ into $L_p[0,1]$. Since $0<\al-a\le b-a$, the first part of Theorem \ref{te1} applies to $\al - a$, hence
\eqref{enRa1} gives
\be
\label{r2}
e_n(R^{\al - a}) \le c_{b-a}\, n^{-(\al-a)}\,.
\ee
In view of \eqref{dec} estimates  \eqref{r1} and \eqref{r2} imply
$$
e_{2n-1}(R^\al)\le e_n(R^a)\,e_n(R^{\al-a})\le C_{a,p,q}\, c_{b-a}\,n^{-\al}
$$
leading to
$$
e_n(R^\al)\le 2^\al \,C_{a,p,q}\, c_{b-a}\,n^{-\al}\le c_{a,b}\,n^{-\al}
$$
with $c_{a,b}= 2^b\,C_{a,p,q}\, c_{b-a}$. This being true for all $a<\al\le b$ proves \eqref{enRa2} and completes the proof of Theorem \ref{te1}.
\hspace*{\fill}$\square$\medskip\par
\medskip

The next corollary of Theorem \ref{te1} will not be used later on. But we believe that it could be of interest in its own right because it shows that also the constants $c_{\al,p,q}$ on the left hand side of \eqref{enRa} may be chosen uniformly.
\begin{cor}
Let $b>(1/p-1/q)_+$ be a given real number. Then there is a constant $\kappa_{b,p,q}>0$ such that for all
$(1/p-1/q)_+<\al\le b$ it follows that
$$
\kappa_{b,p,q}\,n^{-\al}\le e_n(R^\al : L_p[0,1]\to L_q[0,1])\,.
$$
\end{cor}
\begin{proof}
Choose an arbitrary $\al$ with  $(1/p-1/q)_+<\al<b$. By Proposition \ref{enr} we have
$$
c_{b,p,q}\,(2n-1)^{-b}\le e_{2n-1}(R^b)\le e_n(R^\al)\,e_n(R^{b-\al})
$$
where $R^{b-\al}$ is regarded as operator from $L_p[0,1]$ to $L_p[0,1]$ and $R^\al$ as operator
from $L_p[0,1]$ into $L_q[0,1]$. Next we apply Theorem \ref{te1} to $R^{b-\al}$. Note that $0<b-\al<b$, hence
\eqref{enRa1} implies
$$
e_n(R^{b-\al}) \le c_b\, n^{-(b-\al)}\,.
$$
Combining these two estimates leads to
$$
2^{-b}\,c_{b,p,q}\, c_b^{-1}\,n^{-\al}\le e_n(R^\al)
$$
which completes the proof with $\kappa_{b,p,q}=2^{-b}\,c_{b,p,q}\, c_b^{-1}$.
\end{proof}

\section{General entropy bounds for $\Ra$}
\label{Eb}
\setcounter{equation}{0}

\subsection{Upper bounds}
Proposition \ref{enr} asserts that the degree of compactness
of $R^\al$ becomes better along with the growth of the integration order $\al$.
This observation suggests the following:\\
If $\al : [0,1]\to [0,\infty)$ is a measurable function with
\be
\label{al0}
   \al_0:= \inf_{0\le t\le 1}\al(t) > (1/p -1/q)_+\,,
\ee
then the entropy numbers $e_n(\Ra)$ should decrease
at least as fast as  $e_n(R^{\al_0})$.  Our first result says that this is indeed valid.

\begin{prop}
\label{better}
Suppose that $\al_0$, the infimum of $\al(\cdot)$, satisfies $\eqref{al0}$. Then, if $q>1$, it follows
\be
\label{al00}
e_n(\Ra : L_p[0,1]\to L_q[0,1])\le c\, n^{-\al_0}\,.
\ee
If $1< q\le p\le\infty$, the constant $c>0$  in $\eqref{al00}$ may be chosen uniformly for all $\al_0\le b$
while for $1\le p<q\le \infty$ this is valid for all $1/p -1/q < a<b <\infty$ and $\al_0\in[a,b]$. Moreover, in this case it might be that $c>0$ also depends on $p$ and $q$.
\end{prop}

\begin{proof}
Let us start with the slightly more complicated case $p<q$. In view of \eqref{al0} we may choose a number $a$ satisfying $1/p-1/q< a < \al_0$.
Suppose, furthermore, $\al_0\le b$ for a given $b$. Next we take any $\beta\in[a,\al_0)$ and write
$\Ra$ as
\be
\label{decR}
\Ra = R^{\al(\cdot)-\beta}\circ R^\beta
\ee
where $R^\beta : L_p[0,1]\to L_q[0,1]$ and $R^{\al(\cdot)-\beta}$ acts in $L_q[0,1]$.
Because  of $\beta\in[a,b]$ we may apply Theorem \ref{te1}. Consequently, there is a constant $c_{a,b}>0$ (maybe depending also on $p$ and $q$) such that
\be
\label{entb}
e_n(R^\beta: L_p[0,1]\to L_q[0,1])\le c_{a,b}\,n^{-\beta}\,.
\ee
Furthermore, $\inf_{0\le t\le 1}[\al(t)-\beta]>0$, hence, since $q>1$, Theorem \ref{tb} shows
\be
\label{normRb}
\|R^{\al(\cdot)-\beta} : L_q[0,1]\to L_q[0,1]\|\le c_q\,.
\ee
It is important to know that $c_q$ may be taken independent of $\beta$. Combining \eqref{decR}, \eqref{entb} and \eqref{normRb} leads to
$$
e_n(\Ra)\le \|R^{\al(\cdot)-\beta}\|\, e_n(R^\beta)\le c_q\, c_{a,b}\,n^{-\beta}\,.
$$
This being true for all $a\le \beta<\al_0$ allows us to take the limit $\beta\to \al_0$ and proves the proposition for $p<q$.

The case $1<q\le p\le\infty$ follows by the same arguments. The only difference is that here we may choose $\beta$ arbitrarily in $(0,\al_0)$ because in this case the first part of Theorem \ref{te1} applies.
\end{proof}
\medskip

\begin{remark}
If $\al(t)>\al_0$ a.e.~and if $\al(\cdot)-\al_0$ satisfies \eqref{L1n}, then Proposition \ref{better} also holds for $q=1$. Note that in this case we may choose $\beta=\al_0$ . Then Proposition \ref{bL1} applies and leads to $\|R^{\al(\cdot)-\al_0} : L_1[0,1]\to L_1[0,1]\|<\infty$. Writing $\Ra= R^{\al(\cdot)-\al_0}\circ R^{\al_0}$ gives directly the desired estimate.
\end{remark}
\medskip

Suppose now that $\al(\cdot)$ attains its infimum $\al_0$ at a single point. Then it is very likely that the entropy numbers $e_n(\Ra)$ even tend faster to zero than those of
$R^{\al_0}$. We shall investigate this question for increasing functions
$\al(\cdot)$.

\begin{prop}
\label{lem1}
Suppose $1\le p\le\infty$ and $1<q\le\infty$. If $\al(\cdot)$ is non--decreasing so that $\al_0=\al(0)>(1/p-1/q)_+$, then for each $r\in(0,1)$ and integers $n_1$ and $n_2$ it follows that
\be
\label{split}
e_{n_1+n_2-1}(\Ra)\le c_1\, r^{\al(0)+1/q-1/p}\,n_1^{-\al(0)} + e_{n_2}(R^{\al_r(\cdot)})\le
 c_2\, \left(r^{\al(0)+1/q-1/p}\,n_1^{-\al(0)} + n_2^{-\al(r)}\right)
\ee
where
\be
\label{alr}
\al_r(t)=\left\{
\begin{array}{ccc}
\al(t) &:& r\le t\le 1\\
\al(r) &:& 0\le t<r\,.
\end{array}
\right.
\ee
If $p\le q$, the constants $c_1$ and $c_2$ may be chosen independent of $p$ and $q$, only depending on $b>0$ whenever $\al(t)\le b$. In the case $q<p$ the constants $c_1, c_2>0$ probably depend on $p$ and $q$ and may be chosen uniformly for functions $\al(\cdot)$ satisfying $a\le \al(t)\le b$ for some $(1/p-1/q)_+<a<b<\infty$.
\end{prop}
\begin{proof}
Let $P_{[0,r]}$ and $P_{[r,1]}$ be the projections defined by
$$
 P_{[0,r]}:=f\,\ind_{[0,r]}\quad \mbox{and}\quad P_{[r,1]} f := f\,\ind_{[r,1]}\,
$$
respectively. Then we get
$$
\Ra
= P_{[0,r]}\,\Ra + P_{[r,1]}\Ra =  P_{[0,r]}\,\Ra\,P_{[0,r]} + P_{[r,1]}\,R^{\al_r(\cdot)}
$$
with $\al_r(\cdot)$ defined by \eqref{alr}.
Consequently we obtain
\bea
\label{e2n}
\nonumber
e_{n_1+n_2 -1}(\Ra)&\le& e_{n_1}(P_{[0,r]}\,\Ra\,P_{[0,r]})
    + e_{n_2}(P_{[r,1]}\,R^{\al_r(\cdot)})\\
&\le& e_{n_1}(P_{[0,r]}\,\Ra\,P_{[0,r]}) + e_{n_2}(R^{\al_r(\cdot)})\,.
\eea
Observe that $ P_{[0,r]}\,\Ra\,P_{[0,r]}$ is nothing else as
$\Ra$ regarded from $L_p[0,r]$ to $L_q[0,r]$. Hence Proposition \ref{scale}
applies and gives
\be
\label{MR}
e_{n_1}(P_{[0,r]}\,\Ra\,P_{[0,r]})\le \|M_{\al,r}\|\,e_{n_1}(R^{\tilde\al(\cdot)})
\le c\,r^{\al(0)+1/q-1/p}\,n_1^{-\al(0)}
\ee
where the last estimate follows by Proposition \ref{better} because of
$$
\inf_{0\le t\le 1}\tilde\al(t)= \inf_{0\le t\le 1}\al(t\,r) =\al(0)\,.
$$
Plugging \eqref{MR} into \eqref{e2n} proves the first estimate in \eqref{split}.

Another application of Proposition \ref{better}, yet this time with
$\al_r(\cdot)$, finally implies
$$
e_{n_2}(R^{\al_r(\cdot)})\le c\,n_2^{-\al_r(0)} = c\,n_2^{-\al(r)}
$$
and this gives the second estimate in \eqref{split}. This completes the proof.
\end{proof}
\medskip

Let us state now a useful corollary of Proposition \ref{lem1} .
\begin{cor}
\label{co1}
Let $0=r_0<r_1<\cdots<r_m=1$ be a partition of $[0,1]$.
Furthermore, let $n_1,\ldots,n_m$ be given integers and set $N:=\sum_{j=1}^m n_j$.
Then it follows that
\be \label{e_iter}
   e_{N-m+1}(\Ra)
   \le c\,\sum_{j=1}^{m} r_j^{\al(r_{j-1})+1/q-1/p}\, n_j^{-\al(r_{j-1})}\,.
\ee
Here the constant $c>0$ neither depends on the $n_j$ and the integer $m$ nor on the choice of
the partition.
\end{cor}
\begin{proof}
An application of Proposition \ref{lem1} with $r_1$ and for $n_1$ and $\tilde n_2 =\sum_{j=2}^m n_j -m +2$ gives
\be
\label{it1}
e_{N-m+1}(\Ra)=e_{n_1+\tilde n_2-1}(\Ra)\le c\, r_1^{\al(0)+1/q-1/p}\,n_1^{-\al(0)} + e_{\tilde n_2}(R^{\al_{r_1}(\cdot)})\,.
\ee
Recall that $\al_{r_1}(t)=\al (r_1)$ if $0\le t\le r_1$ and $\al_{r_1}(t)=\al(t)$ whenever $r_1\le t\le 1$.

Next we apply again Proposition \ref{lem1} , yet this time with $r_2$ and for $R^{\al_{r_1}(\cdot)}$. Define $\tilde n_3$ by
$\tilde n_3=\sum_{j=3}^m n_j - m +3$.
If $\al_{r_2}$ is given by $\al_{r_2}(t)=\al_{r_1}(r_2)=\al(r_2)$ whenever $0\le t\le r_2$ and $\al_{r_2}(t)=\al(t)$ otherwise, then we get
\be
\label{it2}
e_{\tilde n_2}(R^{\al_{r_1}(\cdot)})=e_{n_2+\tilde n_3-1}(R^{\al_{r_1}(\cdot)})
\le c\, r_2^{\al(r_1)+1/q-1/p}\,n_2^{-\al(r_1)} + e_{\tilde n_3}(R^{\al_{r_2}(\cdot)})\,.
\ee
Plugging \eqref{it2} into \eqref{it1} by $r_0=0$ we obtain
$$
e_{N-m+1}(\Ra)\le c\, r_1^{\al(r_0)+1/q-1/p}\,n_1^{-\al(r_0)} + c\, r_2^{\al(r_1)+1/q-1/p}\,n_2^{-\al(r_1)} + e_{\tilde n_3}(R^{\al_{r_2}(\cdot)})\,.
$$
Proceeding further we end up with
\be
\label{iter}
e_{N-m+1}(\Ra)\le c\,\sum_{j=1}^{m-1}  r_j^{\al(r_{j-1})+1/q-1/p}\,n_j^{-\al(r_{j-1})} + e_{n_m}(R^{\al_{m}(\cdot)})\,.
\ee
But Proposition \ref{better} yields  (recall $r_m=1$)
$$
e_{n_m}(R^{\al_{m}(\cdot)})\le c\, n_m^{-\al(r_{m-1})} = c\, r_m^{\al(r_{m-1})+1/q-1/p}\, n_m^{-\al(r_{m-1})}\,.
$$
Plugging this into \eqref{iter} completes the proof.
\end{proof}

\subsection{Lower bounds}
The basic aim of this subsection is to prove the counterpart of Proposition \ref{lem1} in the case
that $\al(\cdot)$ is bounded from above.
\begin{prop}
\label{lower}
Assume that $\sup_{0\le t\le r} \al(t) \le \al_1$ for some $0<r\le 1$. Then for $n\in\N$ it follows that
\be
\label{lower1}
 e_n(\Ra:L_p[0,r]\to L_q[0,r] ) \ge C n^{-\al_1} r^{\al_1+1/q-1/p}
\ee
for some $C=C(\al_1,p,q)>0$ independent of $n$ and $r$.
\end{prop}

\begin{proof}
We fix $n$ and split $[0,r]$ as
\[
    [0,r]= \bigcup_{j=1}^n I_j, \qquad I_j:=\left[\frac{(j-1)r}{n},\frac{jr}{n}\right].
\]
Introduce the related bases
\[
 \phi_{j,p}:=(n/r)^{1/p}\, \ind_{I_j}, \quad  \phi_{j,q}:=(n/r)^{1/q}\, \ind_{I_j}.
\]
Than we can identify $\ell_p^n$ with $\mathrm{span}\big((\phi_{j,p})_{j=1}^n\big)\subset L_p[0,r] $
and  $\ell_q^n$ with $\mathrm{span}\big((\phi_{j,q})_{j=1}^n\big)\subset L_q[0,r]$.
We also need the averaging operator
\[
   A_n:L_q[0,r]\to \ell_q^n
\]
acting by
\[
   (A_n g) (s):= \frac n r\, \int_{I_j}g(t) \d t, \qquad s\in I_j.
\]
By using $||A_n||\le 1$ we make the first estimate
\be \label{lower2}
    e_n(\Ra:L_p[0,r]\to L_q[0,r] )
    \ge
    e_n(A_n\Ra:L_p[0,r]\to \ell_q^n )
    \ge
    e_n(A_n\Ra: \ell_p^n  \to \ell_q^n ).
\ee
Now we arrived to an operator acting in $n$--dimensional Euclidean space through a triangular matrix
$\sigma:=(\sigma_{ij})_{i,j=1}^n$, i.e.
\[
  A_n\Ra \phi_{j,p}:= \sum_{i=1}^n \sigma_{ij}  \phi_{i,q}
\]
where $\sigma_{ij}=0$ whenever $i<j$ and
\[
  \sigma_{jj}
  = (n/r)^{-1/q} \cdot (n/r)\cdot (n/r)^{1/p} \int_{I_j}
  \frac{1}{\Gamma(\al(t))}\int_{\frac{(j-1)r}{n}}^{t}  (t-u)^{\al(t)-1}\d u\,\d t \,.
\]
We are not interested in evaluation of  $\sigma_{ij}$ whenever $i>j$.
Note that
\begin{eqnarray*}
\sigma_{jj}
  &=& (n/r)^{-1/q+1+1/p} \int_{I_j}
  \frac{1}{\Gamma(\al(t)+1)}   \left( t-\tfrac{(j-1)r}{n}\right)^{\al(t)}\d t
\\
&\ge&
 (n/r)^{-1/q+1+1/p}
  \ \frac{1}{C_1} \cdot \frac{r}{2n}\cdot   \left( \frac{r}{2n}\right)^{\al_1}
\\
&=&
 C_2\  (n/r)^{-1/q+1/p -\al_1}\,,
\end{eqnarray*}
where
\[
  C_1:= \max_{0\le t\le 1}\Gamma(\al(t)+1) \le \max_{1\le a\le \al_1+1}\Gamma(a)
  =\max\{1, \Gamma(\al_1+1)\}.
\]
Now we apply the volumic argument. By the triangular nature of the matrix
$\sigma$ we see that for any Borel set $D\subset \ell_p^n$ it is true that
\[
   \mathrm{vol}_n\left(A_n\Ra(D)\right) \ge \left( C_2\  (n/r)^{-1/q+1/p -\al_1}\right)^n \mathrm{vol}_n(D).
\]
Apply this to $D=B_p^n$, the unit ball of $\ell_p^n$. Assuming that its image
$A_n\Ra(B_p^n)$ is covered by $2^n$ balls of radius $\eps>0$ in $\ell_q^n$ we
get a volumic inequality
\begin{eqnarray*}
\mathrm{vol}_n(B_p^n) &\le&   \left( C_2\  (n/r)^{-1/q+1/p -\al_1}\right)^{-n}  \mathrm{vol}_n\left(A_n\Ra(B_p^n)\right)
\\
 &\le&   \left( C_2\  (n/r)^{-1/q+1/p -\al_1}\right)^{-n}\  2^n \ \eps^n\ \mathrm{vol}_n\left(B_q^n\right).
\end{eqnarray*}
It follows that
\[
 \eps
 \ge \left( \frac{\mathrm{vol}_n\left(B_p^n\right)}{\mathrm{vol}_n\left(B_q^n\right)}
 \right)^{1/n} \ \frac{C_2}{2}\  (n/r)^{-1/q+1/p -\al_1},
\]
By letting $\eps\searrow e_{n+1}(A_n\Ra)$ we also obtain
\[
 e_{n+1}(A_n\Ra)
 \ge \left( \frac{\mathrm{vol}_n\left(B_p^n\right)}{\mathrm{vol}_n\left(B_q^n\right)}
 \right)^{1/n} \ \frac{C_2}{2}\  (n/r)^{-1/q+1/p -\al_1},
\]
Given that (see e.g. \cite{Hu})
\[
 \mathrm{vol}_n\left(B_q^n\right)^{1/n}
  = \frac{2\Gamma\left(1+\frac 1q\right)}{\Gamma\left(\frac nq +1\right)^{1/n}}
  \approx n^{-1/q},
  \quad \textrm{resp.}\quad
  \mathrm{vol}_n\left(B_p^n\right)^{1/n} \approx n^{-1/p},
\]
we obtain the bound
\[
  e_n(A_n\Ra)\ge e_{n+1}(A_n\Ra)\ge  c\ n^{ -\al_1}\,r^{\al_1+1/q-1/p}\,.
\]
It remains to merge it with \eqref{lower2}, and we obtain
the desired bound \eqref{lower1}.
\end{proof}

\begin{remark}
The idea to prove lower entropy bounds by volume estimates for triangular matrices was already used in \cite{L08} for the case $p=2$ and $q=\infty$.
\end{remark}

\section{Examples}
\setcounter{equation}{0}
\label{Exa}


\noindent
{\bf Example 1.} Consider the function
\be
\label{ex1}
    \al(t)=\al_0 +\lambda t^\gamma, \qquad 0\le t\le 1\,,
\ee
for some $\lambda,\gamma>0$ and $\al_0>(1/p-1/q)_+$.
This is the most typical type of behavior around a single critical point.
It was studied, for example, in \cite{DK,HLS}. Here we will prove the following:

\begin{prop}
Let $\al(\cdot)$ be as in $\eqref{ex1}$. Then
there are constants $c,C>0$ such that
\be
\label{ex1a}
\frac{c\,n^{-\al_0}}{(\ln n)^{(\al_0 +1/q-1/p)/\gamma}}\le
 e_n(\Ra : L_p[0,1]\to L_q[0,1]) \le \frac{C\,n^{-\al_0}}{(\ln n)^{(\al_0 +1/q-1/p)/\gamma}}\,.
\ee
\end{prop}

\begin{proof}
Let us start with proving the right hand estimate in \eqref{ex1a}. To this end
we will apply the iterative bound \eqref{e_iter} for estimating the entropy numbers of $\Ra$.
For $n\ge 3$ let $m:=1+[\ln n]$ and set
$$
r_j:=\left\{
\begin{array}{ccc}
\big(\frac{j}{\ln n}\big)^{1/\gamma} &: & 0\le j\le m-1\\
1 &:& j=m
\end{array}
\right.
$$
as well as
\[
  n_j := \left[ \frac{n}{j^2}\right],  \qquad 1\le j\le m.
\]
Clearly, we have
\[
  \sum_{j=1}^m  n_j \le \frac{\pi^2}{6}\, n \le 2n.
\]
Notice also that $n\ge 3$ yields
$n_j=[n/j^2]\ge [n/m^2] = \left[n/(1 + \ln n )^2\right] \ge 1$.

Now we start the evaluation of each term of the sum \eqref{e_iter}.
Because of $r_j\le 1$ we get
$$
  r_j^{\al(r_{j-1})+1/q-1/p}
  \le r_j^{\al_0+1/q-1/p}\le \frac{j^{(\al_0+1/q-1/p)/\gamma}}{(\ln n)^{(\al_0+1/q-1/p)/\gamma}}
\,,\quad 1\le j \le m\,.
$$
On the other hand, with $\al_1:=\sup_{0\le t\le 1}\al(t)=\al_0+\lambda $ we have
\beaa
    n_j^{-\al(r_{j-1})}
    &\le& \left(\frac{n}{2 j^2}\right)^{-\al(r_{j-1})}
    \le (2 j^2)^{\al_1}\,n^{-\al(r_{j-1})}
    \\
    &=& (2 j^2)^{\al_1}\,n^{-\al_0}\, n^{-\lambda (j-1)/\ln n}
    = (2 j^2)^{\al_1}\,n^{-\al_0}\,\ex^{-\lambda (j-1)}\,.
\eeaa
By summing up the bounds it follows that
\beaa
   \sum_{j=1}^m r_j^{\al(r_{j-1})+1/q-1/p}\,n_j^{-\al(r_{j-1})}
   &\le&
   \frac{n^{-\al_0}}{(\ln n)^{(\al_0+1/q-1/p)/\gamma}}\,
   \sum_{j=1}^m (2 j^2)^{\al_1}j^{(\al_0+1/q-1/p)/\gamma}\,\ex^{-\lambda (j-1)}
   \\
   &\le& \frac{C_1\, n^{-\al_0}}{(\ln n)^{(\al_0+1/q-1/p)/\gamma}}
\eeaa
where
\beaa
   C_1&=& C_1(\al_0,\lambda ,\gamma)
   := \sum_{j=1}^\infty (2 j^2)^{\al_1}j^{(\al_0+1/q-1/p)/\gamma}\,\ex^{-\lambda (j-1)}\\
   &=&  2^{\al_0+\lambda}
      \sum_{j=1}^\infty j^{2\al_0+ 2\lambda +(\al_0+1/q-1/p)/\gamma} \, \ex^{-\lambda (j-1)}
\,.
\eeaa
Finally recall that \eqref{e_iter} yields
\[
   e_{N-m+1}(\Ra)
   \le c\,\sum_{j=1}^{m} r_j^{\al(r_{j-1})+1/q-1/p}\, n_j^{-\al(r_{j-1})}
\]
where $N=\sum_{j=1}^m n_j$. In our case we have $N-m+1\le 2 n$, hence
it follows
$$
   e_{2 n}(\Ra)
   \le e_{N-m+1}(\Ra)\le  \frac{C_2 \, n^{-\al_0}}{(\ln n)^{(\al_0+1/q-1/p)/\gamma}}
$$
implying
$$
   e_{n}(\Ra)
   \le   \frac{C \, n^{-\al_0}}{(\ln n)^{(\al_0+1/q-1/p)/\gamma}}
$$
where $C>0$ depends on $\lambda,\gamma$ and $b$ whenever $0<\al_0\le b$.
\medskip

Next we prove the lower estimate in \eqref{ex1a}. Our aim is to apply \eqref{lower1} with
$
r:=\frac{1}{(\ln n)^{1/\gamma}}
$
and $\al_1=\al(r)$.
Since
$$
\al(r)=\al_0 + \frac{\lambda}{\ln n}
$$
it follows
$$
n^{-\al(r)}= \ex^{-\lambda}\,n^{-\al_0}
$$
while
$$
r^{\al(r)+1/q-1/p}= \frac{1}{(\ln n)^{(\al_0+1/q-1/p)/\gamma}}\,
\left(\frac{1}{\ln n}\right)^{\frac{\lambda}{\gamma\,\ln n}}
\ge \frac 1 2\, \frac{1}{(\ln n)^{(\al_0+1/q-1/p)/\gamma}}
$$
for $n$ sufficiently large. Thus Proposition \ref{lower} leads to
\beaa
  \lefteqn{e_n(\Ra : L_p[0,1]\to L_q[0,1])
  \ge e_n(\Ra : L_p[0,r]\to L_q[0,r])}\\&\ge& C\, r^{\al(r)+1/q-1/p}\,n^{-\al(r)}
   \ge \frac{c\,n^{-\al_0}}{(\ln n)^{(\al_0+1/q-1/p)/\gamma}}
\eeaa
as asserted. This completes the proof.
\end{proof}
\medskip

\begin{remark}
Using estimate \eqref{split} in the proof of the upper bound in \eqref{ex1a}  instead of
\eqref{e_iter}, we only get the weaker
$$
    e_n(\Ra) \le C\,  n^{-\al_0} \
    \left( \frac{\ln\ln n } {\ln n}\right )^{(\al_0+1/q-1/p)/\gamma}\,.
$$
This shows that the iteration formula \eqref{e_iter} is in fact necessary
to obtain the right order.
\end{remark}
\medskip

\noindent
{\bf Example 2.}
Consider the function
\be
\label{ex2}
     \al(t)=\al_0+ \lambda |\ln t|_+^{-\gamma}, \qquad 0\le t\le 1\,,
\ee
for some $\lambda,\gamma>0$ and $\al_0>0$.
Here we will prove the following:

\begin{prop}
Let $\al(\cdot)$ be as in $\eqref{ex2}$. Then
there are constants $c,C>0$ such that
\begin{eqnarray}
\nonumber
 \lefteqn{c\,   n^{-\alpha(0)}
\exp \left\{ -  \alpha_0^{\gamma/(1+\gamma)}
\frac{\gamma+1}{\gamma^{\gamma/(1+\gamma)}}
(\lambda \ln n)^{1/(1+\gamma)}(1+o(1))  \right\}}\\
&\le&  e_n(\Ra : L_p[0,1]\to L_p[0,1])
\label{ex2u}
   \le C \, n^{-\al_0}
    \exp\left\{-\al_0^{\gamma/(1+\gamma)}(\lambda \ln n)^{1/(1+\gamma)}\right\}.
\end{eqnarray}
\end{prop}

\begin{proof}
By choosing $\ln r :=- \left( \frac{\lambda \ln n}{\al_0}\right)^{1/(1+\gamma)}$
in \eqref{split} for $n$ we obtain the upper bound in
\eqref{ex2u}.
By choosing
$\ln r := -  \left( \frac{\gamma  \lambda \ln n}{\alpha_0}\right)^{1/(1+\gamma)}$
in \eqref{lower1} we obtain the lower bound in \eqref{ex2u}.
\end{proof}

\begin{remark}
The degree of $\ln n$ under the exponents in \eqref{ex2u} is the same for
for the lower and the upper bound, but the constants are not. It is possible
that this gap can be bridged by using  more delicate estimates like
\eqref{e_iter} and analogous refinements of \eqref{lower1}.
\end{remark}
\medskip

\noindent
{\bf Example 3.} Consider the function
\be
\label{ex3}
     \al(t):=\al_0 + \exp\{- \lambda t^{-\gamma} \}, \qquad 0\le t\le 1\,,
\ee
for some $\lambda,\gamma>0$ and $\al_0>0$.
Here we will prove the following:

\begin{prop}
Let $\al(\cdot)$ be as in $\eqref{ex3}$. Then
there are constants $c,C>0$ such that
\be
\label{ex3a}
   c\,  n^{-\al_0}\left(\ln\ln n\right)^{-\al_0/\gamma}
   \le  e_n(\Ra : L_p[0,1]\to L_p[0,1])
   \le  C\,  n^{-\al_0}\left(\ln\ln n\right)^{-\al_0/\gamma}.
\ee
\end{prop}

\begin{proof}
By choosing $r := \lambda^{1/\gamma}
\left( \ln\left(\frac{\gamma\ln n}{\al_0\ln\ln\ln n}\right)\right)^{-1/\gamma}$
 in \eqref{split} we obtain the upper bound in \eqref{ex3a}.
By choosing  $r:= \lambda^{1/\gamma}( \ln\ln n)^{-1/\gamma}$
in \eqref{lower1} we obtain the lower bound in \eqref{ex3a}.
\end{proof}

\medskip

\noindent
{\bf Example 4.}
In contrast to the previous examples, this one relies not on Proposition \ref{lem1}
and Corollary \ref{co1}, but on Proposition \ref{p1} and Theorem \ref{t1}, respectively.
Let $\al(\cdot)$ be a function on $(0,1]$ with $\inf_{\eps\le t\le 1}\al(t)>0$ for each $\eps>0$.
If
$$
\lim_{t\to 0} \al(t)\,|\ln t| =\infty\,,
$$
then Theorem \ref{t1} implies that the operator $\Ra$ is compact in $L_p[0,1]$, Moreover, by
$$
\Ra = P_{[0,r]}\,\Ra \,P_{[0,r]} + P_{[r,1]}\,\Ra
$$
it follows that
$$
e_n(\Ra)\le \|P_{[0,r]}\,\Ra \,P_{[0,r]}\| + e_n(P_{[r,1]}\,\Ra)\,.
$$
Proposition \ref{p1} tells us that $\|P_{[0,r]}\,\Ra \,P_{[0,r]}\|\le c\,\sup_{0<t\le r} (2 t)^{\al(t)}$. Suppose now that
$\al(\cdot)$ is non--decreasing. Under this assumption Proposition \ref{better} yields
$$
e_n(P_{[r,1]}\,\Ra)\le c\,n^{-\al(r)}\,.
$$
Summing up, we arrive at
\be
\label{enw}
e_n(\Ra)\le c\left(\sup_{0<t\le r} (2 t)^{\al(t)} + n^{-\al(r)}\right)
\ee
for each $0<r\le 1$.

For $0<\gamma<1$ regard now $\al(\cdot)$ defined by
$$
\al(t)=\left\{
\begin{array}{ccc}
|\ln t|^{-\gamma}&:& 0<t\le \ex^{-1}\\
1&:& \ex^{-1}\le t\le 1\,.
\end{array}
\right.
$$
Applying \eqref{enw} with $r=n^{-1}$ leads to
$$
e_n(R^{\al(\cdot)})\le c\, n^{-(\ln n)^{-\gamma}}= c\,\ex^{-(\ln n)^{1-\gamma}}\,.
$$
\bigskip

\noindent
\textbf{Final remarks:}
To the best of our knowledge, in this paper for the first time
continuity and compactness properties of fractional
integration operators with variable order are investigated. Thus it is quite natural that
some important questions remain open. Let us set up a list of the most interesting
ones.

\bee
\item
In view of Propositions \ref{better} and \ref{lower} the following question arises naturally.
Let $\al(\cdot)$ and $\beta(\cdot)$ be two measurable functions such that
$\al(t)\le \beta(t)$ for almost all $t\in[0,1]$. Suppose furthermore
$\al_0=\inf_{0\le t\le 1}\al(t)$ satisfies $\al_0>(1/p-1/q)_+$, hence
$\Ra$ is a compact operator from $L_p[0,1]$ into $L_q[0,1]$. Does this imply
$$
e_n(R^{\beta(\cdot)})\le c\,e_n(\Ra)
$$
with a certain constant $c>0$ only depending on $p$ and $q$ ?
\item
If $1/2<\al<3/2$, then the classical Riemann--Liouville operator $R^\al$ is
tightly related to the fractional Brownian motion $B_H$ where $H=\al-1/2$.
The link between these two objects is the integration operator
$V^\al : L_2(\R)\to L_\infty[0,1]$ defined by
$$
   (V^\al f)(t):= \frac{1}{\Gamma(\al)}\,
   \int_{-\infty}^0\left[(t-s)^{\al-1}-(-s)^{\al-1}\right]\,f(s)\d s\,,
   \quad 0\le t\le 1\,.
$$
More precisely, if $(f_k)_{k\ge 1}$ is an orthonormal bases in $L_2(\R)$, then
with $S^\al=R^\al+V^\al$ it holds
$$
   B_H(t)=c_H\,\sum_{k=1}^\infty \xi_k (S^\al f_k)(t)\,,\quad 0\le t\le 1\,,
$$
where $(\xi_k)_{k\ge 1}$ denotes a sequence of independent standard
normal random variables.

The crucial point in this link is that $V^\al$ has very strong compactness
properties. Namely, as shown in \cite{BL1} and \cite{BL2}, the entropy numbers
$e_n(V^\al)$ tend to zero exponentially. As a consequence, $R^\al$ and
$S^\al$ are quite similar with respect to their compactness properties.
Suppose now that $\al(\cdot)$ is a function with
$$
   1/2<\inf_{0\le t\le 1}\al(t)\le \sup_{0\le t\le 1} \al(t)<3/2\,.
$$
Then $V^{\al(\cdot)}$ is well--defined and one can prove that it is also
bounded as operator from $L_2(\R)$ into $L_\infty[0,1]$. Moreover, as $S^\al$
generates the fractional Brownian motion $B_H$ with $H=\al-1/2$, the operator
$S^{\al(\cdot)}:=\Ra + V^{\al(\cdot)}$ generates the so-called multi-fractional
Brownian motion $B_{H(\cdot)}$ with $H(t)=\al(t)-1/2$, see
\cite{ACL,BJR,FL}. But in order to relate
$\Ra$ and $S^{\al(\cdot)}$, hence $\Ra$ and $B_{H(\cdot)}$, one should know that
the entropy numbers of $V^{\al(\cdot)}$ tend to zero faster than those of $\Ra$.
But at the moment we do not know whether this is true. At least, the methods
used in the classical case do no longer work, some completely new approach
is necessary.
\item
The methods developed in Sections \ref{Eec} and \ref{Eb} lead also to suitable
upper estimates for the approximation numbers $a_n(\Ra)$ at least if
$1\le q\le p\le \infty$. It would be interesting to find such estimates as
well for the remaining cases of $p$ and $q$. Note that in contrast to
$e_n(R^\al)$ the behavior $a_n(R^\al)$ depends heavily on the choice of $p$ and
$q$ (cf.~\cite{ET}).
\eee


\vskip 1pc
\noindent

\parbox[t]{8cm}
{Mikhail Lifshits, \\
Department of Mathematics and Mechanics,\\
St.~Petersburg State University,\\
198504 St.~Petersburg, Russia,\\
email: mikhail@lifshits.org}\hfill
\parbox[t]{6cm}
{Werner Linde\\
Friedrich--Schiller--Universit\"at Jena \\
Institut f\"ur Stochastik\\
Ernst--Abbe--Platz 2\\
07743 Jena,
Germany\\
email: werner.linde@uni-jena.de}

\end{document}